\documentclass[tikz]{nmd/article}
\usepackage{amscd}
\usepackage{floatrow}
\ifnmd
\fi

\title{Certifying the Thurston norm via \\ SL(2,\kern 2ptC)-twisted homology}

\author{Ian Agol}
\givenname{Ian}
\surname{Agol}
\address{ University of California, Berkeley \\
  970 Evans Hall \#3840 \\
  Berkeley, CA, 94720-3840\\
  USA
}
\email{ianagol@berkeley.edu}
\urladdr{http://math.berkeley.edu/~ianagol/}

\author{Nathan M. Dunfield}
\givenname{Nathan}
\surname{Dunfield}
\address{ Dept.~of Math., MC-382 \\
          University of Illinois \\
          1409 W. Green St. \\
          Urbana, IL 61801 \\ 
          USA
}
\email{nathan@dunfield.info}
\urladdr{http://dunfield.info}

\arxivreference{}
\arxivpassword{}

\newcommand{\Fxy}{\langle x, y \rangle}
\newcommand{\Fab}{\pair{a, b}}
\newcommand{\ZFxy}{\Z[\Fxy ]}
\newcommand{\Ea}{E_\alpha}

\newcommand{\Rpm}{R_\pm}
\newcommand{\Rmp}{R_\mp}
\newcommand{\Rp}{R_+}
\newcommand{\Rm}{R_-}
\newcommand{\Mpm}{M_\pm}
\newcommand{\Mmp}{M_\mp}

\newcommand{\Npm}{N_\pm}
\newcommand{\Np}{N_+}
\newcommand{\Nm}{N_-}
\newcommand{\Xtilp}{\Xtil_+}
\newcommand{\Xtilm}{\Xtil_-}
\newcommand{\Xtilpm}{\Xtil_\pm}
\newcommand{\Rbarp}{\Rbar_+}

\newcommand{\KT}[2]{\mathit{KT}_{#1, #2}}

\DeclareMathOperator{\interior}{int}

\begin{document}

\begin{dedication} 
In memory of Bill Thurston: his amazing mathematics will live on, \\
but as a collaborator, mentor, and friend he is sorely missed.  
\end{dedication}

\begin{abstract} 
  We study when the Thurston norm is detected by twisted Alexander
  polynomials associated to representations of the 3-manifold group to
  $\SL{2}{\C}$.  Specifically, we show that the hyperbolic torsion
  polynomial determines the genus for a large class of hyperbolic
  knots in $S^3$ which includes all special arborescent knots and many
  knots whose ordinary Alexander polynomial is trivial.  This theorem
  follows from results showing that the tautness of certain sutured
  manifolds can be certified by checking that they are a product from
  the point of view of homology with coefficients twisted by an
  $\SL{2}{\C}$-representation.
\end{abstract}
\maketitle

\section{Introduction}

For a compact orientable \3-manifold $M$, the Thurston norm on
$H_2(M, \partial M; \Z) \cong H^1(M; \Z)$ measures the minimal
topological complexity of a surface representing a particular homology
class.  Twisted Alexander polynomials are a powerful tool for studying
the Thurston norm; such a polynomial $\tau(M, \phi, \alpha)$ depends
on a class $\phi \in H^1(M; \Z)$ and a representation
$\alpha \maps \pi_1(M) \to \GL{}(V)$, where $V$ is a
finite-dimensional vector space over a field $K$.  The polynomial
$\tau(M, \phi, \alpha)$ is constructed from the homology with
coefficients twisted by $\alpha$ of the cyclic cover of $M$ associated
to $\phi$.  The degree of any such
$\tau(M, \phi, \alpha) \in K[t^{\pm1}]$ gives a lower bound on the
Thurston norm of $\phi$ \cite{FriedlKim2008}.  Remarkably, Friedl and
Vidussi \cite{FriedlVidussi2012} showed that given $M$ and $\phi$ one
can always choose $\alpha$ so that this lower bound is sharp, with the
possible exception of when $M$ is a closed graph manifold; their
results rely on the fact that most Haken \3-manifold groups are full
of cubulated goodness \cite{Wise2012, Liu2013, PrzytyckiWise2012,
  PrzytyckiWise2013} so that \cite{Agol2008} applies.

Here, we explore whether one can get sharp bounds from just
representations to $\SL{2}{\C}$, especially those that originate in
a hyperbolic structure on $M$. When $M$ is the exterior of a
hyperbolic knot $K$ in $S^3$, there is a well-defined \emph{hyperbolic
  torsion polynomial} $\cT_K \in \C[t^{\pm 1}]$ which is (a
refinement of) the twisted Alexander polynomial associated to a lift
to $\SL{2}{\C}$ of the holonomy representation
$\pi_1(M) \to \Isom^+(\H^3) = \PSL{2}{\C}$.  The experimental
evidence in \cite{DunfieldFriedlJackson2012} forcefully led to

\begin{conjecture}[\cite{DunfieldFriedlJackson2012}] 
  \label{conj:knots}
  For a hyperbolic knot in $S^3$, the hyperbolic torsion polynomial
  determines the Seifert genus $g(K)$; precisely, $\deg{\cT_K} = 4
  g(K) - 2$.
\end{conjecture}

\noindent 
Here, we prove this conjecture for a large class of knots, which
includes infinitely many knots whose ordinary Alexander
polynomial is trivial.  We call a knot $K \subset S^3$
\emph{libroid} if there is a collection $\Sigma$ of disjointly embedded
minimal genus Seifert surfaces in its exterior $X = S^3 \setminus
N(K)$ so that $X \setminus \Sigma$ is a union
of books of $I$-bundles in a way that respects the structure of $X
\setminus \Sigma$ as a sutured manifold; see Section~\ref{sec:libdef}
for the precise definitions.   We show
\begin{restatable*}{theorem}{theoremsharpknots}\label{thm:sharpknots}
  Conjecture~\ref{conj:knots} holds for libroid hyperbolic knots in $S^3$.
\end{restatable*}

\noindent
Libroid knots generalize the notion of a fibroid surface introduced in
\cite{CullerShalen1994}, and includes all fibered knots.  The class of
libroid knots is closed under Murasugi sum (Lemma~\ref{lem:murasugi})
and contains all \emph{special} arborescent knots obtained from
plumbing oriented bands (this includes \2-bridge knots), as well as
many knots whose ordinary Alexander polynomial is trivial
(Theorem~\ref{thm:hyplib}).  Previous to Theorem~\ref{thm:sharpknots},
Conjecture~\ref{conj:knots} was known only in the case of 2-bridge
knots, by work of Morifuji and Tran \cite{Morifuji2012,
  MorifujiTran2014}.

\subsection{Motivation} While twisted Alexander polynomials give sharp
bounds on the Thurston norm if one allows arbitrary representations to
$\GL{n}{\C}$ by \cite{FriedlVidussi2012}, there are still compelling
reasons to consider questions such as Conjecture~\ref{conj:knots}.
First, if the Thurston norm is detected by representations of
uniformly bounded degree, then one should be able to use ideas from
\cite{Kuperberg2014} to show that the KNOT GENUS problem of
\cite{AgolHassThurston2006} is in $\NP \cap \coNP$ for knots in $S^3$
using a finite-field version of $\tau(M, \phi, \alpha)$ as the $\coNP$
certificate.  (As with the results in \cite{Kuperberg2014}, this would
be conditional on the Generalized Riemann Hypothesis.)  Second, since
$\cT_K$ is easily computable in practice, a proof of
Conjecture~\ref{conj:knots} should lead to an effectively
polynomial-time algorithm for computing $g(K)$ for knots in $S^3$.
Finally, Conjecture~\ref{conj:knots} would be another beautiful
Thurstonian connection between the topology and geometry of
3-manifolds.

\subsection{Sutured manifolds}  

The Thurston norm bounds associated to twisted Alexander polynomials
can be understood in the following framework of \cite{FriedlKim2013}.
Throughout, see Section~\ref{sec:objects} for precise definitions.
Let $M = (M, R_-, R_+, \gamma)$ be a sutured manifold.  Given a
representation $\alpha \maps \pi_1(M) \to \GL{}(V)$, we say that $M$
is an $\alpha$-homology product if the inclusion-induced maps
\[
H_*(R_+; \Ea) \to H_*(M;  \Ea) \mtext{and} H_*(R_-; \Ea) \to H_*(M, \Ea)
\]
are all isomorphisms; here $\Ea$ is the system of local coefficients
associated to $\alpha$.  An $\alpha$-homology product is necessarily a
taut sutured manifold (see Theorem~\ref{thm:Eprodistaut} for the
precise statement).  Conversely, every taut sutured manifold is an
$\alpha$-homology product for \emph{some} representation $\alpha$ by
\cite{FriedlKim2013}.  A
weaker, less geometric, parallel to Conjecture~\ref{conj:knots} is
\begin{conjecture}\label{conj:sutured}
  For a taut sutured manifold $M$, there exists $\alpha
  \maps \pi_1(M) \to \SL{2}{\C}$ for which $M$ is a homology product.  
\end{conjecture} 
Theorem~\ref{thm:sharpknots} will follow easily from the next result,
establishing a strong version of Conjecture~\ref{conj:sutured} for
books of $I$-bundles (see Section~\ref{sec:books} for the definitions).
\begin{restatable*}{theorem}{theoremtorusguts}\label{thm:torusguts}
  Let $M$ be a taut sutured manifold which is a book of
  $I$-bundles.  Suppose $\alpha \maps \pi_1(M) \to \SL{2}{\C}$ has
  $\tr\big(\alpha(\gamma)\big) \neq 2$ for every curve $\gamma$ which is the
  core of a gluing annulus for an $I$-bundle page.  Then $M$ is an
  $\alpha$-homology product.
\end{restatable*}
\noindent
In trying to attack Conjecture~\ref{conj:sutured}, an intriguing
aspect of Theorem~\ref{thm:torusguts} is the very weak hypotheses on
the representation $\alpha$.  Unfortunately, for more complicated taut
sutured manifolds one must put additional restrictions on $\alpha$ to
get a homology product, as the next result shows.
\begin{restatable*}{theorem}{theoremsuturedexample}\label{thm:example}
  There exists a taut sutured manifold $M$ with a faithful
  discrete and purely hyperbolic representation
  $\alpha \maps \pi_1(M) \to \SL{2}{\C}$ where $M$ is \emph{not} an
  $\alpha$-homology product.  The manifold $M$ is acylindrical with
  respect to the pared locus consisting of the sutures.
\end{restatable*}
\noindent 
Another instance where we can prove Conjecture~\ref{conj:sutured} is
\begin{restatable*}{theorem}{theoremhandlebody}
  \label{thm:handlebody}
  Suppose $M$ is a sutured manifold which is a genus 2 handlebody with
  suture set $\gamma$ a single curve separating $\partial M$ into two
  once-punctured tori.  If the pared manifold $(M, \gamma)$ is
  acylindrical and $M \setminus \gamma$ is incompressible, then $M$ is
  a homology product with respect to some 
  $\alpha \maps \pi_1(M) \to \SL{2}{\C}$.  
\end{restatable*}
\noindent
With both Theorems~\ref{thm:torusguts} and \ref{thm:handlebody}, it is
easy to construct sutured manifolds satisfying their hypotheses which are
\emph{not} homology products with respected to $H_*( \ \cdot \ ; \Q)$.

\subsection{Outline of contents} 

After reviewing the needed definitions in Section~\ref{sec:objects},
we establish the basic properties of homology products in
Section~\ref{sec:basics} and so relate Conjectures~\ref{conj:knots}
and \ref{conj:sutured}.  Section~\ref{sec:products} is devoted to
proving Conjecture~\ref{conj:sutured} in the two cases mentioned
above.  Section~\ref{sec:example} studies one sutured manifold in
detail, characterizing which $\SL{2}{\C}$-representations make it a
homology product (Theorem~\ref{thm:locus}); Theorem~\ref{thm:example}
is an easy consequence of this.  Finally, Section~\ref{sec:libroid} is
devoted to studying libroid knots, both showing that this is a large
class of knots and also proving Theorem~\ref{thm:sharpknots} follows
from Theorem~\ref{thm:torusguts}.

\subsection{Acknowledgements} 

Agol was partially supported by the Simons Foundation and US NSF
grants DMS-1105738 and DMS-1406301.  Dunfield was partially supported
by US NSF grant DMS-1106476, a Simons Fellowship, the GEAR Network (US
NSF grant DMS-1107452) and this work was partially completed while
visiting ICERM (Brown University) and the University of Melbourne.  We
thank Stefan Friedl for several helpful discussions, and we are very
grateful to the referee for an extraordinarily quick and yet very 
thorough review of this paper.

\section{Background}\label{sec:objects}

We begin with the precise definitions of the basic objects we will be
working with.  Throughout, all manifolds will be assumed orientable
and moreover oriented. 

\subsection{Taut surfaces} For a connected surface, define
$\eulerminus{S} = \max\big( -\chi(S), 0 \big)$; extend this to all
surfaces via
$\eulerminus{S \sqcup S'} = \eulerminus{S} + \eulerminus{S'}$.  For a
\3-manifold $M$ and a (possibly empty) subsurface $A \subset \bdry M$,
the Thurston norm of $z \in H_2(M, A; \Z)$ is defined by
\[
\norm{ z } = \min \setdef{ \eulerminus{S} }{\mbox{$S$ is a properly
    embedded surface representing $z$ with $\bdry S \subset A$}}
\]
A properly embedded compact surface $S$ in a
\3-manifold $M$ is \emph{taut} if $S$ is incompressible and
realizes the Thurston norm for the class $[S, \partial S]$ in
$H_2\big(M, N(\partial S) ; \Z \big)$.

\subsection{Sutured manifolds} A sutured manifold
$(M, R_+, R_-, \gamma)$ is a compact \3-manifold with a partition of
$\partial M$ into two subsurfaces $R_+$ and $R_-$ along their common
boundary $\gamma$.  The surface $R_+$ is oriented by the
outward-pointing normal, and $R_-$ is oriented by the inward pointing
one.  Note that the orientations of $\Rpm$ induce a common orientation
on $\gamma$.  A sutured manifold is \emph{taut} if it is irreducible
and the surfaces $\Rpm$ are both taut.  A connected sutured manifold
$M$ is \emph{balanced} if it is irreducible, $\chi(R_-) = \chi(R_+)$,
not a solid torus with $\gamma = \emptyset$, and if any component of
$\Rpm$ has positive $\chi > 0$ then $M$ is $D^3$ with a single suture.
A disconnected sutured manifold is balanced if each connected
component is.  Note that any taut sutured manifold is necessarily
balanced.

\subsection{Notes on conventions}  
We follow \cite{Scharlemann1989} in requiring taut surfaces to be
incompressible; this is not universal, and the difference is just that
the more restrictive definition excludes a solid torus with no sutures
and a ball with more than one suture. Like \cite{FriedlKim2013} but
unlike many sources, we do not allow torus sutures consisting of an
entire torus component of $\partial M$.  Our definition of balanced is
slightly more restrictive than that of \cite{FriedlKim2013} and also
differs from that of \cite{Juhasz2008}.

\subsection{Twisted homology}\label{sec:twisted_homology}

Suppose $X$ is a connected CW complex with a representation
$\alpha \maps \pi_1(X) \to \GL{}(V)$, where $V$ is a vector space over
a field $K$.  Let $\Ea$ be the system of local coefficients over $X$
corresponding to $\alpha$; precisely, $\Ea \to X$ is the induced
vector bundle where we give each fiber the \emph{discrete topology} so
that $\Ea \to X$ is actually a covering map.  (Alternatively, you can
view $\Ea$ as an ordinary vector bundle equipped with a flat
connection.)  Throughout, we use the geometric definition of homology
with local coefficients $H_*(X; \Ea)$ given in
\cite[pg.~330--336]{Hatcher2002} which does not require a choice of
basepoint; it is equivalent to the more algebraic definition of
e.g.~\cite[pg.~328--330]{Hatcher2002}.  More generally, if $X$ is not
connected, we can consider a bundle $E \to X$ with fiber $V$ and the
associated homology $H_*(X; E)$.  We also use the geometrically
defined cohomology $H^*(X; E)$ of \cite[pg.~333]{Hatcher2002}. Of
course, both $H_*(X; E)$ and $H^*(X; E)$ satisfy all the usual
properties: a relative version for $(X, A)$, long exact sequence of a
pair, Mayer-Vietoris, etc.

If $X$ is a compact oriented $n$-manifold with $\partial X$
partitioned into two submanifolds with common boundary $A$ and $B$
then one has Poincar\'e duality:
\begin{equation}\label{eq:poincare}
D_M \maps H^k(X, A; E) \xrightarrow{\ \cong \ } H_{n-k}(X, B; E)
\end{equation}
where $D_M$ is given by cap product with the ordinary relative
fundamental class $[X, \partial X] \in H_n(X, \partial X; \Z)$.

Let $E^* \to X$ denote the bundle where we have replaced each fiber
with its dual vector space; for $\Ea$, this corresponds to using the
dual or contragredient representation
$\alpha^* \maps \pi_1(X) \to \GL{}{(V^*)}$ defined by
$\alpha^*(g) = \big(\rho(g^{-1})\big)^*$.  When $X$ has finitely many
cells, the relevant version of universal coefficients is that
$H_k(X; E) \cong H^k(X ; E^*)$ as $K$-vector spaces.

When $E^* \cong E$ as bundles over $X$, we say that $E$ is self-dual.
Examples include $\Ea$ where $\alpha \maps \pi_1(X) \to \SL{2}{K}$;
specifically, $\alpha^*$ is conjugate to $\alpha$ via
$\mysmallmatrix{0}{1}{-1}{0}$.  Seen another way, the action of
$\SL{2}{K}$ on $K^2$ preserves the standard symplectic form
$x_1 y_2 - x_2 y_1$ and hence $\Ea$ has a nondegenerate inner product on
each fiber allowing us to identify $\Ea$ with $\Ea^*$.
Representations that are unitary with respect to some involution on
$K$ may not be self-dual, but still satisfy
$H_*(X, A; \Ea) \cong H^*(X, A; \Ea)$, as $K$-vector spaces, for any
$A \subset X$; we call such representations/bundles
\emph{homologically self-dual}.

\section{Basics of twisted homology products}\label{sec:basics}

Throughout this section, $E$ will be a system of local coefficients
over a sutured manifold $M$ with fiber a vector space of dimension
$n \geq 1$.  As in the introduction, we say that $M$ is an
$E$-\emph{homology product} if the inclusion induced maps
$H_*(\Rpm; E) \to H_*(M; E)$ are both isomorphisms.  This is
equivalent to the notion of an $E$\hyp \emph{cohomology product} where
$H^*(M; E) \to H^*(\Rpm; E)$ are isomorphisms: the former is the same
as $H_*(M, \Rpm; E) = 0$, the latter is the same as
$H^*(M, \Rpm; E) = 0$, and by Poincar\'e duality one has
$H_{k}(M, \Rpm; E) \cong H^{3-k}(M, R_\mp; E)$.  These concepts are
parallel to \cite{FriedlKim2013}, where they consider unitary
representations of balanced sutured manifolds where
$H_1(M, \Rm; \Ea) = 0$ because of:
\begin{proposition}\label{prop:homprodcond}
  Suppose $M$ is a connected balanced sutured manifold with both
  $\Rpm$ nonempty.  If $E$ is homologically self-dual, then $M$ is an
  $E$-homology product if and only if any one of the following eight
  groups vanish: $H_k(M, \Rpm; E)$ and $H^k(M, R_\pm; E)$ for
  $1 \leq k \leq 2$.
\end{proposition}

\begin{proof}
  Since both of $\Rpm$ are nonempty, it follows that
  $H_0(M, \Rpm; E) = H^0(M, \Rpm; E) = 0$, and so by Poincar\'e
  duality we have $H_3(M, \Rmp; E) = 0$.  We focus on the case where
  $H_1(M, \Rm; E) = 0$; the other cases are similar. Since $M$ is
  balanced, we have $\chi(\Rm) = \chi(M)$ and hence
  $\chi\big(H_*(M, \Rm; E)\big) = 0$.  Since we know that
  $H_k(M, \Rm; E) = 0$ for every $k \neq 2$, this forces
  $H_2(M, \Rm; E) = 0$ as well.  By Poincar\'e, we have
  $H^*(M, \Rp; E) = 0$.  Since $E$ is homologically self-dual, this
  gives $H_*(M, \Rp; E) = 0$, and so $M$ is an $E$-homology product as
  claimed.
\end{proof}

\noindent
Our motivation for studying twisted homology products is the following
two results:

\begin{theorem}[{\cite[\S 3]{FriedlKim2013}}]
  \label{thm:Eprodistaut}
  Suppose $M$ is an irreducible sutured manifold which is an $E$\hyp
  homology product and where no component of $M$ is a solid torus
  without sutures.  Then $M$ is taut.
\end{theorem} 

\begin{theorem}[{\cite[\S 4]{FriedlKim2013}}]
  \label{thm:torsionconn}
  Suppose $X$ is a compact irreducible \3-manifold with $\partial X$ a
  (possibly empty) union of tori.  For $\phi \in H^1(X; \Z)$
  nontrivial and $\alpha \maps \pi_1(X) \to \GL{}{(V)}$, the torsion
  polynomial $\tau(X, \phi, \alpha)$ gives a sharp lower bound on the
  Thurston norm $\norm{\phi}$ if and only if when $S$ is a taut
  surface without nugatory tori dual to $\phi$ the sutured manifold
  $M$ which is $X$ cut along $S$ is an $\alpha$-homology product.
\end{theorem}
\noindent
Here, a set of torus components of a taut surface $S$ are
\emph{nugatory} if they collectively bound a submanifold of $X$
disjoint from $\partial X$.  Theorem~\ref{thm:Eprodistaut} is explicit
and Theorem~\ref{thm:torsionconn} is implicit in Sections 3 and 4 of
\cite{FriedlKim2013} respectively; however, to make this paper more
self-contained, we include proofs of both results.

Let $S$ be a properly embedded compact surface in a \3-manifold $N$;
we do not assume either $S$ or $N$ is connected, and $N$ is allowed to
be noncompact and have boundary.  We say $S$ \emph{separates $N$ into
  $\Np$ and $\Nm$} if $N = \Np \cup_S \Nm$, the positive side of $S$
is contained in $\Np$, the negative side of $S$ is contained in $\Nm$,
and every component of $\Npm$ meets $S$.  The linchpin for
Theorems~\ref{thm:Eprodistaut} and \ref{thm:torsionconn} is the
following lemma, where all homology groups are with respect to some
system $E$ of local coefficients on $N$, and all maps on homology are
induced by inclusion:

\begin{lemma}\label{lem:linchpin}
  Suppose $S$ separates $N$ into $\Npm$.  If both
  $H_*(\Npm) \to H_*(N)$ are surjective then so are
  $H_*(S) \to H_*(\Npm)$ and $H_*(S) \to H_*(N)$. Moreover, if for
  some $k$ both $H_k(\Npm) \to H_k(N)$ are isomorphisms then so are
  $H_k(S) \to H_k(\Npm)$ and $H_k(S) \to H_k(N)$.
\end{lemma}

\begin{proof}
  Since both $H_*(\Npm) \to H_*(N)$ are surjective, the Mayer-Vietoris
  sequence for $N = \Np \cup_S \Nm$ splits into short exact sequences
  \begin{equation}\label{eq:MV}
    0 \to H_k(S) \xrightarrow{i_+ \oplus i_-} H_k(\Np) \oplus H_k(\Nm)\xrightarrow{j_+-j_-} H_k(N) \to 0   
  \end{equation}
  To see that $H_k(S) \to H_k(\Np)$ is surjective, take
  $c_+ \in H_k(\Np)$ and choose $c_- \in H_k(\Nm)$ which maps to the
  same element in $H_k(N)$ as $c_+$; then $(c_+, c_-) \mapsto 0$ under
  $j_+-j_-$ and hence $c_+$ is the image of some element of
  $H_k(S)$ by exactness of (\ref{eq:MV}). Symmetrically, $H_k(S) \to
  H_k(\Nm)$ is also surjective, proving the first part of the lemma. 

  Suppose in addition that both $H_k(\Npm) \cong H_k(N)$.  Since $S$
  is compact and $H_k(S)$ surjects $H_k(\Npm)$ and $H_k(N)$, it
  follows that all four $K$-vector spaces are finite\hyp dimensional.
  Since $H_k(\Npm)\cong H_k(N)$, exactness of (\ref{eq:MV}) forces
  $H_k(S) \cong H_k(N)$, and hence the surjections
  $H_k(S) \to H_k(\Npm)$ must be isomorphisms as claimed.
\end{proof}

\noindent
We now show that a sutured manifold which is a homology product must be
taut.  

\begin{proof}[Proof of Theorem~\ref{thm:Eprodistaut}]
  We may assume that $M$ is connected.  All homology groups will have
  coefficients in $E$ unless otherwise indicated, and let $n$ be
  the dimension of the fiber of $E$.  We first reduce to the case
  where every component of $\Rpm$ has $\chi \leq 0$.  If a component
  of $\Rpm$ is a sphere, then $M$ must be $D^3$ by irreducibility with
  (say) $R_+ = \partial M$ and $R_- = \emptyset$.  Since $E$ must be
  trivial over $D^3$, we get that $\dim(H_0(M))= n$. However, $H_0(\Rm) = 0$
  contradicting that $M$ is an $E$-homology product.  If some
  component of $\Rpm$ is a disc, say $D \subset R_+$, then
  $\dim( H_0(D) ) = n$ and since a connected space will have $H_0$ of
  dimension at most $n$, we conclude that $\dim( H_0(M) ) = n$.
  However, then $E$ must be the trivial bundle, since nontrivial
  monodromy around some loop would reduce $\dim( H_0(M) )$ below $n$.
  It follows that $M$ is also a homology product with respect to
  $H_*( \ \cdot \ ; K)$, and hence $\Rpm$ are both connected and thus
  discs; by irreducibility, $M$ is $D^3$ with one suture and hence
  taut.  So from now on we assume that every component of $\Rpm$ has
  $\chi \leq 0$.

  Since we have excluded $M$ from being a solid torus with no sutures,
  all of the torus components of $\Rpm$ are incompressible.  Thus to
  prove that $M$ is taut it remains to show that $\Rpm$ realize the
  Thurston norm of their common class in $H_2(M, N(\gamma); \Z)$.
  Note this is automatic if $\Rp = \emptyset$ since the homology
  product condition implies $\chi(\Rm) = 0$, so from now on we assume
  both $\Rpm$ are nonempty.  Suppose $S$ is any other surface in that
  homology class.  Throwing away components of $S$ that bound
  submanifolds of $M$ that are disjoint from $\partial M$, we can
  assume that $S$ separates $M$ into $\Mpm$, where each $\Mpm$
  contains $\Rpm$ respectively. We next show that the theorem
  follows from:
  \begin{claim}\label{claim:hom}
    The maps $H_k(\Rpm) \to H_k(\Mpm)$ are isomorphisms for $k \neq 1$
    and injective for $k = 1$. The maps $H_k(S) \to H_k(\Mpm)$ are 
    isomorphisms for $k \neq 1$ and surjective for $k = 1$.
  \end{claim}
  From the claim we get that $H_k(S) \cong H_k(\Rpm)$ for $k \neq
  1$ and $\dim H_1(S) \geq \dim H_1(\Rpm)$; hence
  \[
  n \cdot \chi(S) = \chi\big( H_*(S) \big) \leq \chi\big( H_*(\Rpm)
  \big) = n \cdot\chi(\Rpm)
  \]
  and so 
  \[
  \eulerminus{S} \geq -\chi(S) \geq -\chi(\Rpm) = \eulerminus{\Rpm}.
  \]
  Thus $\Rpm$ must realize the Thurston norm in its class, establishing
  the proposition modulo Claim~\ref{claim:hom}. 

  To prove the claim, first note that $S$, $\Rpm$, and $\Mpm$ are all
  homotopy equivalent to 2-complexes and so we need only consider
  $k \leq 2$.  Since $\Rpm \hookrightarrow M$ gives isomorphisms on
  $H_*$, we know $H_*(\Rpm) \to H_*(\Mpm)$ is injective and
  $H_*(\Mpm) \to H_*(M)$ is surjective.  Since every component of
  $\Mpm$ meets $\Rpm$, it follows that $H_0(\Rpm) \to H_0(\Mpm)$ is
  onto and hence an isomorphism; consequently, so is
  $H_0(\Mpm) \to H_0(M)$.  Since $H_*(M, \Rpm) = 0$, the long exact
  sequence of the triple $(M, \Mpm, \Rpm)$ gives that
  $H_2(\Mpm, \Rpm) \cong H_3(M, \Mpm)$; by excision and Poincar\'e
  duality, we have $H_3(M, \Mpm) \cong H_3(\Mmp, S) \cong H^0(\Mmp, \Rmp)$
  and the latter vanishes since each component of $\Mmp$ meets
  $\Rmp$.  Thus we have shown $H_2(\Mpm, \Rpm) = 0$, and hence
  $H_2(\Rpm) \to H_2(\Mpm)$ is an isomorphism.

  By Lemma~\ref{lem:linchpin}, we know that each
  $H_*(S) \to H_*(\Mpm)$ is surjective and moreover is an isomorphism
  for $* = 0$.  To see that $H_2(S) \to H_2(\Mpm)$ is an injection
  (and hence an isomorphism), just note that
  $H_3(\Mpm, S) \cong H^0(\Mpm, \Rpm) = 0$.  This proves the claim and
  thus the theorem.
\end{proof}

The last part of this section is devoted to proving the relationship
between the homology product condition and the Thurston norm bounds
coming from twisted torsion/Alexander polynomials.

\begin{proof}[Proof of Theorem~\ref{thm:torsionconn}]
  Let $n = \dim V$.  All homology groups will have coefficients in
  $\Ea$.  Let $\Xtil$ denote the infinite cyclic cover of $X$
  corresponding to $\phi$; it has the structure of a $\Z$'s worth of
  copies of $M$ stacked end to end so that the $\Rp$ on one block is
  glued to the $\Rm$ on the next. (Note that if $\phi$ is not
  primitive then $\Xtil$ is disconnected; you can reduce to the case
  of $\phi$ primitive to avoid this issue if you prefer.)  Let $\Stil$
  be a lift of $S$ to $\Xtil$ corresponding to the top of a preferred
  copy of $M$ in $\Xtil$, and note that $\Stil$ separates $\Xtil$ into
  $\Xtilp$ and $\Xtilm$ which consist of the blocks ``above'' and
  ``below'' $S$ respectively.

  Unwinding the definitions, the precise form of the lower bound given
  in Theorem 14 of \cite{FriedlVidussi2011} (which is Theorem 6.6 in
  the arXiv version) is equivalent to
  \begin{equation}\label{eq:bound}
  \norm{\phi} \geq \frac{1}{n} \left( \dim H_1( \Xtil )  - \dim
    H_0( \Xtil ) - \dim H_2( \Xtil) \right)
  \end{equation}
  where if $H_1( \Xtil)$ is infinite-dimensional the convention
  is to declare the right-hand side as 0.  (When $H_1( \Xtil)$ is
  finite-dimensional so is $H_2(\Xtil)$, see
  e.g.~\cite{FriedlVidussi2011}.)

  The only if direction is easy: if $M$ is an $\alpha$-homology
  product, the Mayer-Vietoris sequence and the fact that homology is
  compactly supported imply that $\Stil \hookrightarrow \Xtil$ gives an
  isomorphism on $H_*$; thus one has 
  \begin{equation} \label{eq:prodgood}
  \norm{\phi} = \eulerminus S \geq -\chi(S) = - \frac{1}{n} \chi\big(
  H_*(S) \big) = - \frac{1}{n} \chi\big( H_*(\Xtil) \big) = \mbox{RHS of (\ref{eq:bound})}
  \end{equation}
  where we have used that $H_3(\Xtil)$ must be 0 since $\Xtil$ is
  noncompact. 

  Conversely, suppose that (\ref{eq:bound}) is sharp.  We will show:
  \begin{claim}
    The maps $H_*(\Stil) \to H_*(\Xtilpm) \to H_*(\Xtil)$ are all isomorphisms. 
  \end{claim}
  \noindent
  The claim implies the theorem as follows: if we take $\Xtilm'$ to be
  $\Xtilm$ shifted down by one, we have
  $\Xtil = \Xtilm' \cup_{\Stil'} M \cup_{\Stil} \Xtilp$.  Applying the
  Mayer-Vietoris sequence to this decomposition, the claim gives that
  $H_*(M) \to H_*(\Xtil)$ is an isomorphism.  Again by the claim, the
  inclusions of $\Stil = \Rp$ and $\Stil' = \Rm$ into $M$ induce
  isomorphisms on $H_*$, and so $M$ is an homology product.

  To prove the claim, begin by noting that $H_*(\Xtil)$ is
  finitely generated, and the $\Z$-action on $\Xtil$ can take any
  particular generating set to one which lies entirely in $\Xtilp$;
  hence $H_*(\Xtilp) \to H_*(\Xtil)$ is onto, as is
  $H_*(\Xtilm) \to H_*(\Xtil)$.  By Lemma~\ref{lem:linchpin}, we know
  $H_*(\Stil) \to H_*(\Xtilpm)$ is onto and an isomorphism when
  $* = 2$ since
  $H_3(\Xtil, \Xtilpm) \cong H_3(\Xtil_\mp, \Stil) \cong 0$ since
  (each component of) $\Xtil_\mp$ is noncompact.  For $* = 0$, we can
  build a compact subset $A$ of $\Xtilpm$ so that
  $H_0(A) \to H_0(\Xtilpm)$ is onto and $H_0(A) \to H_0(\Xtil)$ is an
  isomorphism; consequently, $H_0(\Xtilpm) \to H_0(\Xtil)$ is an
  isomorphism and hence so is $H_0(S) \to H_0(\Xtilpm)$ by
  Lemma~\ref{lem:linchpin}. Finally, from (\ref{eq:prodgood}), we see
  that $\chi(H_*(S)) = \chi(H_*(X))$ and hence the surjection
  $H_1(S) \to H_1(X)$ must be an isomorphism, proving the claim and
  thus the theorem. 
  \end{proof}

\section{Some homology products}\label{sec:products}

This section is devoted to the proof of Conjecture~\ref{conj:sutured}
in two nontrivial cases, both of which include many examples which
are not $\Q$-homology products:  

\theoremtorusguts

\theoremhandlebody

\subsection{Books of \emph{I}-bundles}\label{sec:books}

Recall that a \emph{book of $I$-bundles} is a \3-manifold built from
solid tori (the \emph{bindings}) and $I$-bundles over possibly
nonorientable compact surfaces (the \emph{pages}) glued in the
following way.  For a page $P$ which is an $I$-bundle over a surface
$S$, the \emph{vertical annuli} are the components of the preimage of
$\partial S$. One is allowed to glue such a vertical annulus to any
homotopically essential annulus in the boundary of the binding.  We do
not require that all vertical annuli are glued; those that are not are
called \emph{free}.  For a page $P$, the \emph{vertical boundary}
$\partial_v P$ is the union of all the vertical annuli; the
\emph{horizontal boundary} $\partial_h P$ is
$\partial P \setminus \partial_v P$. We say a sutured manifold is a
book of $I$-bundles if the underlying manifold has such a description
where the sutures are exactly the cores of the free vertical annuli.

\begin{lemma}\label{lem:Ibundle}
  If $M$ is a taut sutured manifold which is a book of $I$-bundles,
  then it has such a structure where all the pages are product
  $I$-bundles.  If the base surface of a page $P$ is not an annulus,
  then one component of the horizontal boundary is contained in $\Rp$
  and the other contained in $\Rm$.  The cores of the vertical annuli
  in the alternate description are homotopic to those in the original
  one.
\end{lemma}

\begin{proof}
  Suppose some page $P$ is a twisted $I$-bundle over a connected
  nonorientable surface $S$.  Then the horizontal boundary
  $\partial_h P$ is connected and hence contained entirely in one of
  $\Rpm$, say $\Rp$.  Then
  $(\Rp \setminus \partial_h P) \cup \partial_v P$ is a surface
  homologous to $\Rp$ with Euler characteristic
  $\chi(\Rp) - 2 \chi(S)$.  Since $\Rp$ is taut, we must have that $S$
  is a M\"obius band.  The pair $(P, \partial_v P)$ is
  homeomorphic to a solid torus $B$ with an annulus that represents
  twice a generator of $\pi_1(B)$.  Thus we can replace $P$ with a
  product bundle over the annulus to which we have attached a new
  component of the binding.

  If a page $P$ is a product $I$-bundle over an orientable surface $S$,
  the same argument shows that if $\partial_h P$ is contained in just
  one of $\Rp$ and $\Rm$ then the base surface must be an annulus. 
  This proves the lemma.
\end{proof}

\noindent
The proof of Theorem~\ref{thm:torusguts} rests on the following simple
observation.
\begin{lemma}\label{lem:circle}
  Suppose $\alpha: \pi_1(S^1) \to \SL{2}{\C}$ is such that $\tr\big(
  \alpha(\gamma)\big) \neq 2$ where $\gamma$ is a generator of
  $\pi_1(S^1)$.  Then $H_*(S^1; \Ea) = 0$. 
\end{lemma}
\begin{proof}
  As with any space, $H_0(S^1; \Ea)$ is the set of co-invariants of
  $\alpha$, that is, the quotient of $\C^2$ by
  $\setdef{\alpha(g) v - v}{g \in \pi_1(S^1), v \in \C^2}$.  If
  $\alpha(\gamma)$ is diagonalizable, then this is 0 since neither
  eigenvalue of $\alpha(\gamma)$ can be $1$ by the trace condition;
  alternatively, if $\alpha(\gamma)$ is parabolic then by the trace
  assumption it is conjugate to $\mysmallmatrix{-1}{1}{0}{-1}$ and
  again the co-invariants vanish.  Since
  $0 = 2 \chi(S^1) = \chi\big( H_*(S^1; \Ea) \big)$ it follows that
  $H_1(S^1; \Ea) = 0$ as well, proving the lemma.
\end{proof}

\noindent
We next establish the first main result of this section.

\begin{proof}[Proof of Theorem \ref{thm:torusguts}] 
  As usual, all homology will have coefficients in $\Ea$.  Consider
  the decomposition of $M$ into $B \cup_A P$, where $B$ is the
  binding, $P$ is the union of all the pages, and $A$ is the union of
  attaching annuli.  By Lemma~\ref{lem:Ibundle}, we can assume that
  $P = (S \times [-1, 1]) \cup Y$ where
  $S \times \{\pm 1\} \subset \Rpm$ and $Y$ is a union of
  $(\mbox{annulus}) \times I$.  By our hypothesis on $\alpha$,
  Lemma~\ref{lem:circle} implies that $H_*(A) = 0$ and $H_*(Y) = 0$.
  Moreover, $H_*(B) = 0$ since the generator of
  $\pi_1( \mbox{component of $B$} )$ has a power which has
  $\tr(\alpha) \neq 2$ and hence must have $\tr(\alpha) \neq 2$ as
  well.  Set $B' = B \cup Y$ and let $A' \subset A$ be the interface
  between $B'$ and $S \times [-1, 1]$.  Applying Mayer-Vietoris to the
  decomposition $M = B' \cup_{A'} (S \times [-1, 1])$ immediately
  gives that $H_*(S \times [-1, 1] ) \to H_*(M)$ is an isomorphism.
  The same reasoning shows that $H_*(S \times \pm 1) \to H_*(\Rpm)$
  are isomorphisms.  Combining, we get that $H_*(\Rpm) \to H_*(M)$
  are isomorphisms, and so $M$ is an $\alpha$-homology product as
  claimed.
\end{proof}

\subsection{Acylindrical sutured handlebodies}

We turn now to the proof of Theorem~\ref{thm:handlebody}.  The following is an immediate consequence of the results in \cite{MenalFerrerPorti2012}.

\begin{theorem} 
  \label{thm:hypprod}
  Suppose $M$ is a sutured manifold where each component of $\Rpm$ is
  a torus.  If the interior of $M$ has a complete hyperbolic metric of
  finite volume, then there exists a lift
  $\alpha \maps \pi_1(M) \to \SL{2}{\C}$ of its holonomy
  representation so that $H_*(M; \Ea) = 0$ and $H_*(\Rpm; \Ea) = 0$.
  In particular, $M$ is an $\alpha$-homology product.
\end{theorem}

\begin{proof}
  By Lemma 3.9 of \cite{MenalFerrerPorti2012}, there is a lift
  $\alpha$ of the holonomy representation so that for each component
  of $\partial M$ there is some curve $c$ with
  $\tr \big(\alpha(c)\big) = -2$. Corollary 3.6 of
  \cite{MenalFerrerPorti2012} now implies that
  $H^*(\partial M; \Ea) = 0$, and Theorem 0.1 of
  \cite{MenalFerrerPorti2012} then gives that $H^*(M; \Ea) = 0$ as well.
  Since $\Ea$ is self-dual, it follows that
  $H_*(\partial M; \Ea) = H_*(M; \Ea) = 0$; since $H_*(\partial M;
  \Ea) = H_*(\Rm; \Ea) \oplus H_*(\Rp; \Ea)$ we are done.  
\end{proof}

\begin{lemma}\label{lem:2handle}
  Let $M$ be a sutured manifold and $N$ be the sutured manifold
  resulting from attaching a \2-handle to $M$ along a component of the
  suture set $\gamma$.  Let $E$ be a system of local coefficients on
  $N$.  If $N$ is an $E$-homology product then $M$ is an $E
  |_M$-homology product.
\end{lemma}
\noindent
This is a natural result since if $N$ is taut then so is $M$, though
the converse is not always true.  

\begin{proof}
  Throughout, all homology is with coefficients in $E$.  Let $\Rbarp
  \subset \partial N$ be the extension of $\Rp$ to the new sutured
  manifold $N$.  Note that $\Rbarp  = \Rp \cup D^2$ and $N = M \cup (D^2
  \times I)$.  Consider the associated Mayer-Vietoris sequences and
  natural maps:
  \[
  \begin{CD}
    @>>> H_k( S^1)  @>>> H_k(\Rp) \oplus H_k(D^2) @>>> H_k(\Rbarp) @>>> \\ 
   @.      @VVV      @VVV      @VVV   @.  \\ 
    @>>> H_k( S^1 \times I)  @>>> H_k(M) \oplus H_k(D^2 \times I) @>>> H_k(N) @>>> 
  \end{CD}
  \]
  The leftmost vertical arrow is an isomorphism since it comes from
  a homotopy equivalence.  The rightmost vertical arrow is an
  isomorphism by hypothesis.  By the five lemma, the middle arrow must
  be an isomorphism; since it is the direct sum of the maps
  $H_k(\Rp) \to H_k(M)$ and $H_k(D^2) \to H_k(D^2 \times I)$ we
  conclude that $H_k(\Rp) \to H_k(M)$ is an isomorphism.  The
  symmetric argument proves that $H_k(\Rm) \to H_k(M)$ is an
  isomorphism for every $k$ and so $M$ is indeed an $E$-homology
  product.
\end{proof}

\begin{theorem} \label{thm:torihomprod}
  Suppose that $M$ is a sutured manifold so that each component of
  $\Rpm$ is a (possibly) punctured torus.  If adding \2-handles to $M$
  along all the sutures results in a hyperbolic manifold, then there
  exists $\alpha \maps \pi_1(M) \to \SL{2}{\C}$ so that $M$ is an
  $\alpha$-homology product.

\end{theorem}
\begin{proof}
  Let $N$ be the result of adding \2-handles to the sutures of $M$.  Let
  $\alpha \maps \pi_1(N) \to \SL{2}{\C}$ be the lift of the holonomy
  representation of the hyperbolic structure on $N$ given by
  Theorem~\ref{thm:hypprod}.  Applying
  Lemma~\ref{lem:2handle} inductively shows that $M$ is a homology
  product with respect to the induced representation $\pi_1(M) \to
  \SL{2}{\C}$ as needed. 
\end{proof}

\noindent
We can now prove the other main result of this section.  

\begin{proof}[Proof of Theorem~\ref{thm:handlebody}]
  By Theorem~\ref{thm:torihomprod} it suffices to prove that the
  result $M_\gamma$ of attaching a \2-handle to $M$ along $\gamma$ is
  hyperbolic.  Being a handlebody, $M$ is irreducible and atoroidal.
  Since $\partial M$ is compressible and $M \setminus \gamma$ is
  incompressible, Theorems A, 1, and 2 of \cite{EudaveMunoz1994} together
  imply that $M_\gamma$ is irreducible, acylindrical, atoroidal, and
  has incompressible boundary (when applying Theorems 1 and 2, note
  that $\gamma$ is separating, which is one of the special cases
  mentioned in the final paragraph of the statements of these
  results).  Thus $\interior(M_\gamma)$ has a complete hyperbolic
  metric of finite-volume as needed.
\end{proof}

\begin{remark}
  The representation $\alpha$ given in the proof of
  Theorem~\ref{thm:handlebody} may seem a bit unnatural since it is
  reducible on $\pi_1(\Rpm)$.  However, it can be perturbed to $\beta$
  for which $M$ is still a homology product and where $\beta$ is
  parabolic free on $\pi_1(M)$ and hence faithful.  The point is just
  that the set of all such $\beta$ is the complement of a countable
  union of \emph{proper} Zariski closed subsets in the character
  variety $X(M) \cong \C^3$, and hence is dense in $X(M)$. 
  Specifically, as discussed in Section~\ref{sec:example}, the locus
  where $M$ is not a homology product is Zariski closed, as of course
  is the set where a fixed nontrivial $\gamma \in \pi_1(M)$ is parabolic.
\end{remark}

\section{An example}\label{sec:example}

Suppose $M$ is a balanced sutured manifold which is homeomorphic to a
genus 2 handlebody.  Assuming that each of $\Rpm$ is connected, then
either $\Rpm$ are both tori with one boundary component or both pairs
of pants.  In this section, we compute $H^1(M, R_+; \Ea)$ in a
specific example as $\alpha$ varies over the $\SL{2}{\C}$ character
variety of $\pi_1(M)$, and so characterize the $\alpha$ for which $M$
is an $\alpha$-homology product.  This leads to the proof of
Theorem~\ref{thm:example} which was discussed in the introduction.

\subsection{Basic setup}

Both $\pi_1(R_+)$ and $\pi_1(M)$ are free groups of rank two, say
generated by $\Fxy$ and $\Fab$ respectively; let $i_* \maps \pi_1(M)
\to \pi_1(R_+)$ be the map induced by the inclusion $i \maps R_+
\hookrightarrow M$.  For $w \in
\Fxy$ we denote its Fox derivatives in $\ZFxy$ by
$\partial_x w$ and $\partial_y w$, where 
\[
\partial_x x = 1, \quad \partial_x x^{-1} = - x^{-1}, \quad \partial_x
y^{\pm 1} = 0, \quad \mbox{and} \quad \partial_x(w_1 \cdot w_2)
= \partial_x w_1 + w_1 \cdot \partial_x w_2
\]
Now fix a representation $\alpha \maps \pi_1(M) \to \GL{}{(V)}$ where
$\dim(V) = 2$, and extend to a ring homomorphism $\alpha \maps \mathbb{Z}[\pi_1(M)]\to \End(V)$.  
\begin{proposition}\label{prop:fox}
  The sutured manifold $M$ is an $\alpha$-homology product precisely
  when the $4 \times 4$ matrix
  \[
  \renewcommand*{\arraystretch}{1.5}
  \left( \begin{array}{cc} \alpha\big( \partial_x i_*(a) \big) & \alpha\big(\partial_y i_*(a)\big)\\
       \alpha\big(\partial_x i_*(b)\big) &  \alpha\big(\partial_y i_*(b)\big)
       \end{array}\right)
  \]
  has nonzero determinant.
\end{proposition}

\begin{proof}
  Consider the 2-complex $W$ with one vertex $v$, four edges $e_x,
  e_y, e_a, e_b$, and two faces $r_a, r_b$ with attaching maps
  specified by the words $i_*(a) \cdot a^{-1}$ and $i_*(b) \cdot b^{-1}$.  For
  the subcomplex $B = e_a \cup e_b$, there is a map $j \maps
  (W, B) \to (M, R_+)$ which induces homotopy equivalences $W \to M$
  and $B \to R_+$ corresponding to the natural maps on fundamental
  groups ($[e_x] \mapsto x$, $[e_a] \mapsto a$, etc.).
  By the long-exact sequence of the pair and the five lemma,
  it follows that $j_*$ induces an isomorphism $ H^*(M, R_+; \Ea) \to H^*(W, B; E_{\alpha
    \circ j_*})$.

  By Proposition~\ref{prop:homprodcond}, to show $M$ is an
  $\alpha$-homology product, it remains to show
  $H^1(W, B; E_{\alpha \circ j_*}) = 0$. As a left module over
  $\Lambda = \ZFxy$, the chain complex of the universal cover $\Wtil$
  of $W$ has the form:
  \[
  C_*\big(\Wtil; \Z\big): \quad 
  0 \to \Lambda r_a \oplus \Lambda r_b \overset{\partial_2}\longrightarrow 
  \Lambda e_x \oplus \Lambda e_y \oplus \Lambda e_a \oplus \Lambda e_b
  \overset{\partial_1}\longrightarrow  \Lambda v \to 0
  \]
  Since $\Lambda$ is noncommutative, it is most natural to write the
  matrices $[\partial_i]$ for the left-module maps $\partial_i$ so
  that they act on row vectors to their left, that is
  $\partial_i(v) = v \cdot [\partial_i]$. In this form, we have the
  following, where we have denoted $i_*(a)$ and $i_*(b)$ in
  $\Fxy$ by just $a$ and $b$:
  \[
 \renewcommand*{\arraystretch}{1.1}
  [\partial_1] = \left( \begin{array}{c} 
      x - 1 \\ y - 1 \\ a - 1 \\ b - 1 
    \end{array}\right) \qquad
  [\partial_2] = \left( \begin{array}{cccc}
      \partial_x a  & \partial_y a  &  -1 & 0\\ 
      \partial_x b  & \partial_y b  & 0 & -1 \\
    \end{array} \right)
  \]
  Applying the functor $\Hom(\, \cdot \,, V_\alpha)$ to get
  $C^*(W; E_{\alpha \circ j_*})$ has the effect of replacing each copy
  of $\Lambda$ with $V$, where the matrices of the coboundary maps $d^i$
  are the result of applying $\alpha \maps \Lambda \to \End(V)$ entrywise
  to the $[\partial_i]$; here the matrices $[d^i]$ act on column
  vectors to their right.  Restricting to the subcomplex of cochains
  vanishing on $B$ gives:
  \[
   C^*(W, B; E_{\alpha \circ j_*}): \quad 0 \leftarrow V^2 \overset{d^1}
   \longleftarrow V^2 \leftarrow 0 \leftarrow 0
  \]
  where $d^1$ is precisely the matrix in the statement of the
  proposition; the result follows.  
\end{proof}

\begin{figure}[tb]
  \floatbox[{\capbeside\thisfloatsetup{capbesideposition={right,center},capbesidewidth=9cm}}]{figure}[\FBwidth]
  {\caption{The sutured manifold $M$ sketched at left is $D^3$ with
      open neighborhoods of the two dark arcs removed, where $\Rp$ and
      $\Rm$ are the pairs of pants indicated. The manifold $M$ is
      homeomorphic to a handlebody, with $\pi_1(M)$ freely generated
      by the loops $x$ and $y$; the element $u$ in $\pi_1(M)$ is
      $y x y x^{-1}y^{-1}$.  These claims can be checked by a
      straightforward calculation starting with a Reidemeister-like
      presentation for $\pi_1(M)$. }
  \label{fig:sutured}}%
{\begin{tikzoverlay}[width=4.6cm]{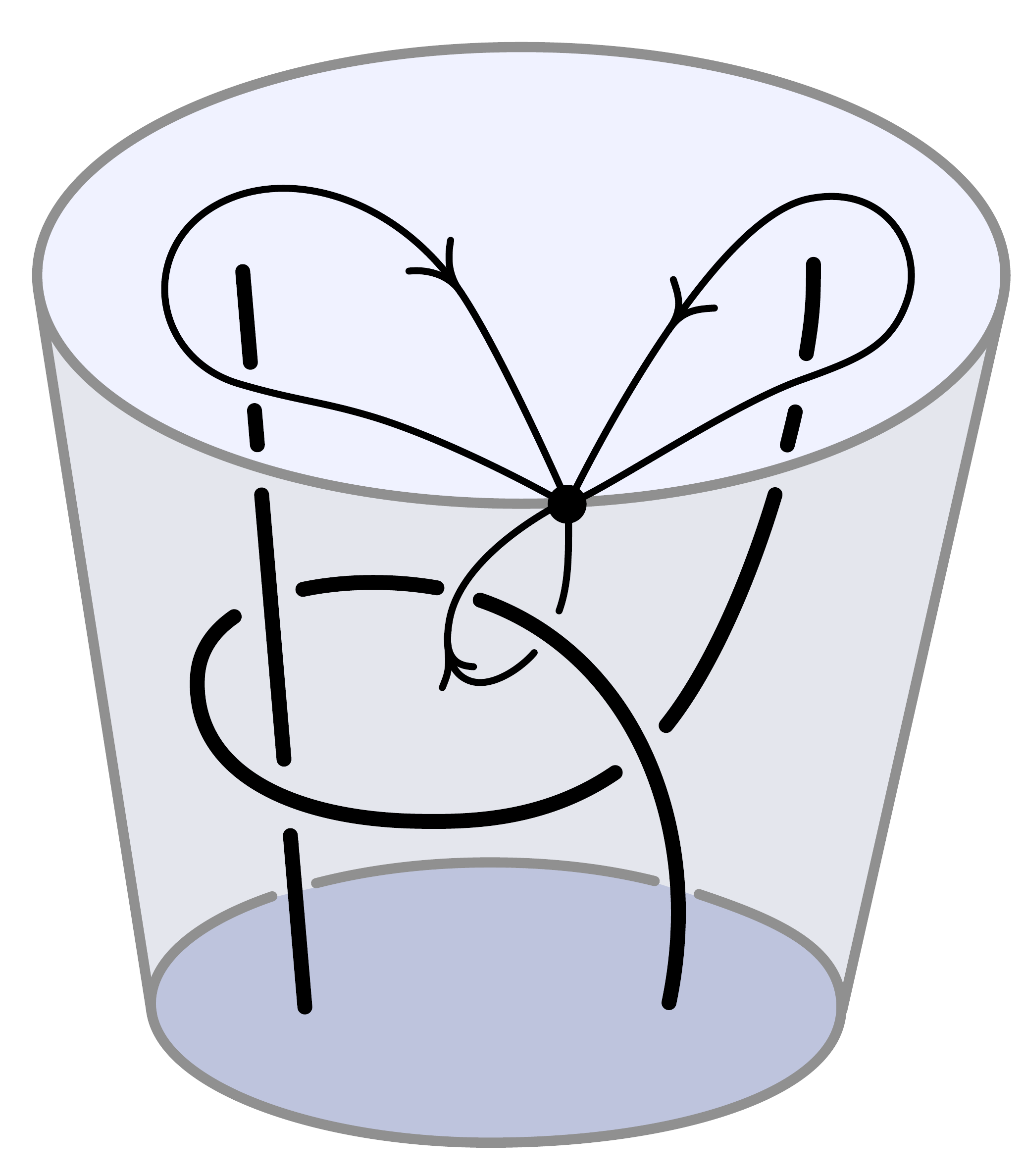}[font=\small]
    \node[] at (53.9,99.7) {$\Rp$};
    \node[] at (50.4,13.4) {$\Rm$};
    \node[below left] at (43.9,88.1) {$x$};
    \node[below left] at (43.9,50.8) {$y$};
    \node[below] at (68.8,84.5) {$u$};
\end{tikzoverlay}}
\end{figure}
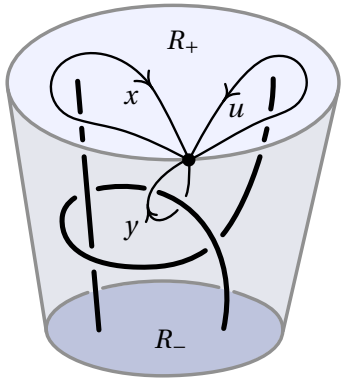

\subsection{Pants example}  Let $M$ be the sutured manifold shown in
Figure~\ref{fig:sutured}, where the free group $\pi_1(\Rp)$ has generators
\[
\pi_1(\Rp) = \pair{x, \,   y x y x^{-1}y^{-1}} 
\]
Let $X(M)$ be the $\SL{2}{\C}$ character variety of $\pi_1(M) =
\Fxy$. Now $X(M) \cong \C^3$ with coordinates $\{\xbar, \ybar, \zbar\}$
corresponding to the trace functions of $\{x, y, xy\}$.  Despite the
fact that $M$ is a product with respect to ordinary $\Z$ homology, we
will show:
\begin{theorem}\label{thm:locus}
  The locus $L$ of $[\alpha] \in X(M)$ where $M$ is not an
  $\alpha$-homology product is a (complex) 2-dimensional plane, 
  namely $\big\{ \xbar  + \ybar - \zbar = 3 \big\}$.
\end{theorem}
\begin{remark}
  Unlike for irreducible representations, characters
  $[\alpha] \in X(M)$ consisting of \emph{reducible} representations
  may contain nonconjugate representations.  For such classes, there
  is thus ambiguity in which local system $E$ to associate with
  $[\alpha]$.  However, it turns out that whether $M$ is an
  $E$-homology product is independent of this choice.
  Similar to \cite[Lemma 7.1]{DunfieldFriedlJackson2012}, the point is
  that reducible representations with the same character share the
  same diagonal part and one uses this with Proposition~\ref{prop:fox}
  to verify the claim; since our focus is on irreducible
  representations, we leave the details to the interested reader.
\end{remark}

\begin{proof}
By Proposition~\ref{prop:fox}, we are interested in when 
\[
\det \left( \begin{array}{cc} \alpha\big( 1 \big) & 0\\
     \alpha\big(y - y x y x^{-1}\big) &  \alpha\big(1 + yx -  y x y x^{-1}y^{-1}\big)
       \end{array}\right) = 0
\]
or equivalently when $\det\big(\alpha(w)\big) = 0$ for $w = 1 + xy
- x y x^{-1} \in \ZFxy$.  Any
irreducible $\alpha$ can be conjugated so that 
\[
\alpha(x) = \left(\begin{array}{cc} 0 & 1 \\  -1 &
    \xbar \end{array}\right) \mtext{and}
\alpha(y) = \left(\begin{array}{cc} \ybar & -u\\  u^{-1} &
    0\end{array}\right)
\mtext{where $u + u^{-1} =\zbar$.}
\]
Applying this $\alpha$ to $w$ and eliminating variables yields that
$\det\big(\alpha(w)\big) = 0$ if and only if $\xbar  + \ybar - \zbar -
3=0$; thus $L$ is as claimed.  
\end{proof}

One representation in $L$ is $(\xbar, \ybar, \zbar) = (4,  4, 5)$ which
can be realized by 
\[
\alpha(x) = \left(\begin{array}{cc} 1 & 1 \\  2 & 3
    \end{array}\right) \mtext{and}
\alpha(y) = \left(\begin{array}{cc} 1 & -2 \\  -1 &
    3\end{array}\right).
\]
An easy calculation shows that the axes of these hyperbolic elements
cross in $\H^2$; since $\alpha \left(xyx^{-1}y^{-1}\right)$ is also
hyperbolic with negative trace, it follows that $\alpha\big(\Fxy\big)$
is a Fuchsian Schottky group \cite{Purzitsky1972}.  In particular,
  $\alpha$ is discrete, faithful, and purely hyperbolic.  This proves:

\theoremsuturedexample

\begin{remark}
  Representations that cover the same homomorphism
  $\pi_1(M) \to \PSL{2}{\C}$ need not give rise to isomorphic
  cohomology.  For example, the Schottky representation above covers
  the same $\PSL{2}{\C}$ representation as $\beta$ where
  $(\xbar, \ybar, \zbar) = (-4, 4, -5)$, which is not in $L$, and hence
  $M$ is a $\beta$-homology product.  In fact, in this example, every
  irreducible representation to $\PSL{2}{\C}$ has \emph{some} lift to
  $\SL{2}{\C}$ for which $M$ is a homology product.
\end{remark}

\begin{remark}
  For each $N \geq 2$, the group $\SL{2}{\C}$ has a unique irreducible
  $N$\hyp dimensional complex representation, which we denote
  $\iota_N \maps \SL{2}{\C} \to \SL{N}{\C}$.  Let $L_N$ be the locus
  of $[\alpha]$ in $X(M)$ where $M$ is not an $\iota_N \circ \alpha$
  homology product.  A straightforward calculation with Gr\"obner
  bases finds:
  \begin{align*}
  L_3 =& \big\{2\xbar\ybar\zbar- \xbar^2 - \ybar^2 - 3\zbar^2 + 3 = 0 \big\} \\ 
  L_4 =& \big\{3\xbar^2\ybar^2\zbar - 3\xbar^2\ybar\zbar^2 - 3\xbar\ybar^2\zbar^2 + \xbar^4 - 2\xbar^3\ybar
  - 2\xbar\ybar^3 + \ybar^4 + 2\xbar^3\zbar \\ 
  & \quad \quad + 3\xbar^2\ybar\zbar +
  3\xbar\ybar^2\zbar + 2\ybar^3\zbar -
  3\xbar\ybar\zbar^2 + 2\xbar\zbar^3 + 2\ybar\zbar^3 + \zbar^4 -
  3\xbar^3 - 3\ybar^3 \\
  & \quad \quad  + 3\zbar^3- 3\xbar^2
  + 6\xbar\ybar - 3\ybar^2 - 6\xbar\zbar - 6\ybar\zbar - 3\zbar^2 +
  6\xbar + 6\ybar - 6\zbar + 9 = 0\}
  \end{align*}
  The intersection $L_2 \cap L_3 \cap L_4$ is zero-dimensional, as one
  would expect from the intersection of three (complex) surfaces in
  $\C^3$.   Computing out a bit farther, we found that $\bigcap_{N =
    2}^{5} L_N = \bigcap_{N =
    2}^{10} L_N$ contains a single point $(\xbar, \ybar, \zbar) = (2,2,1)$ outside
  the reducible representations; in particular, there are no purely hyperbolic
  representations in this intersection.  
\end{remark}

\section{Libroid Seifert surfaces}\label{sec:libroid}

In this last section, we study libroid knots, a notion generalizing
fibered knots and fibroid surfaces which is defined in
Section~\ref{sec:libdef} below.  We will show that this is a large class
of knots for which Conjecture~\ref{conj:knots} holds:

\begin{theorem}\label{thm:hyplib}
  All special arborescent knots, except the $(2,n)$--torus knots, are
  hyperbolic libroid knots.  Moreover, there are infinitely many
  hyperbolic libroid knots whose ordinary Alexander polynomial is
  trivial.
\end{theorem}

\theoremsharpknots 

\subsection{Library sutured manifolds}\label{sec:libdef}

We call a taut sutured manifold $(M, \Rpm, \gamma)$ a \emph{library}
if there is a taut surface
$(\Sigma,\partial \Sigma)\subset (M, N(\gamma))$ such that
$[\Sigma]=n[\Rp]\in H_2(M,N(\gamma); \Z)$ for some $n \geq 0$, and the
sutured manifold $M\setminus \Sigma$ is a book of $I$-bundles in the
sense of Section~\ref{sec:books}.  Note that $M \setminus \Sigma$ has
at least $n+1$ connected components, and thus is a collection of books
of $I$-bundles, that is, a ``library''.  We say that a taut surface
$S \subset X^3$ is a \emph{libroid surface} if $X \setminus S$ is a
library sutured manifold. This generalizes the notion of a
\emph{fibroid surface} \cite{CullerShalen1994}, and in fact the
surface $S \cup \Sigma$ is a fibroid surface.  We say that a knot in
$S^3$ is libroid if it has a minimal genus Seifert surface which is
libroid.  Definitions in hand, we now deduce
Theorem~\ref{thm:sharpknots} from Theorem~\ref{thm:torusguts}.

\begin{proof}[Proof of Theorem~\ref{thm:sharpknots}]
  Let $K$ be a libroid knot with $X$ its exterior, and let
  $\alpha \maps \pi_1(X) \to \SL{2}{\C}$ be a lift of the holonomy
  representation of the hyperbolic structure on $X$.  Let $S$ be a
  minimal genus Seifert surface for $K$ which is libroid.  By
  Theorem~\ref{thm:torsionconn}, we just need to show that the sutured
  manifold $M = X \setminus S$ is an $\alpha$-homology product.  This
  is immediate if $M$ is a product, so we will assume from now on that
  $X$ is not fibered.  Let $\{ \Sigma_i \}$ be disjoint minimal genus
  Seifert surfaces cutting $M$ up into sutured manifolds that are each
  a book of $I$-bundles; for notational convenience, set
  $\Sigma_0 = S$.  It is enough to show that each such book $B$ is an
  $\alpha$-homology product, since they are stacked one atop another
  to form $M$.  To apply Theorem~\ref{thm:torusguts}, we need to check that
  no core $\gamma$ of a gluing annulus has
  $\tr\big(\alpha(\gamma)\big) = 2$.  Assume $\gamma$ is such a
  core, so in particular $\alpha(\gamma)$ is parabolic.

  First note that $\gamma$ is isotopic to an essential curve in some
  $\Sigma_i$.  Since $\Sigma_i$ is minimal genus and not a fiber, by
  Fenley \cite{Fenley1998} it is a quasi-Fuchsian surface in $X$ and
  in particular the only embedded curve in $\Sigma_i$ whose image
  under $\alpha$ is parabolic is $\partial \Sigma_i$, which is the
  homological longitude $\lambda \in \pi_1(\partial X)$. But by
  \cite[Corollary 2.6]{Calegari2006} or \cite[Corollary
  3.11]{MenalFerrerPorti2012}, one always has
  $\tr\big(\alpha(\lambda)\big) = -2$, which contradicts that
  $\alpha(\gamma)$ has trace $+2$.  So we can apply
  Theorem~\ref{thm:torusguts} as desired, proving the theorem.
\end{proof}

\subsection{A plethora of libroid knots}

We now turn to showing that there are many hyperbolic libroid knots.
A key tool for this will be the notion of Murasugi sum, which we
quickly review.  Consider two oriented surfaces with boundary $S_1$
and $S_2$ in $S^3$, and let $L_i = \partial S_i \subset S^3$ be the
associated links. Suppose that $S_1$ and $S_2$ intersect so that there
is a sphere $S^2 \subset S^3$ with $S^3=B_1\cup_{S^2} B_2$, so that
$S_i \subset B_i$; see Figure \ref{fig:murasugi}(a, b), where the
interface between $B_1$ and $B_2$ is a horizontal plane separating
$S_1$ and $S_2$, which have been pulled apart slightly for clarity.
Moreover, assume that $S_1\cap S_2 = P \subset S^2$ is a $2k$-sided
polygon, where the edges of $\partial P$ are cyclically numbered so
that the odd edges lie in $L_1$, and the even edges lie in
$L_2$. Also, assume that the orientations of $S_1$ and $S_2$ agree on
$P$.  Let $L = \partial( S_1\cup S_2)$ be the link obtained as a
boundary of the union of the two surfaces. Then $L$ is said to be
obtained by \emph{Murasugi sum} from $L_1$ and $L_2$. If $k=1$, this
is connected sum, and if $k=2$, then this operation is known as
\emph{plumbing}.  There are two natural Seifert surfaces for $L$ shown
in Figure \ref{fig:murasugi}(b, c), given by $S= S_1\cup S_2$, and
$S'= \big((S_1\cup S_2) - P \big)\cup \overline{(S^2-P)}$. Note that
$S'$ is also a Murasugi sum of the surfaces
$(S_i-P)\cup \overline{S^2-P}$, which are isotopic to $S_i$.

\begin{figure}[tb]
\begin{tikzoverlay}[width=13cm]{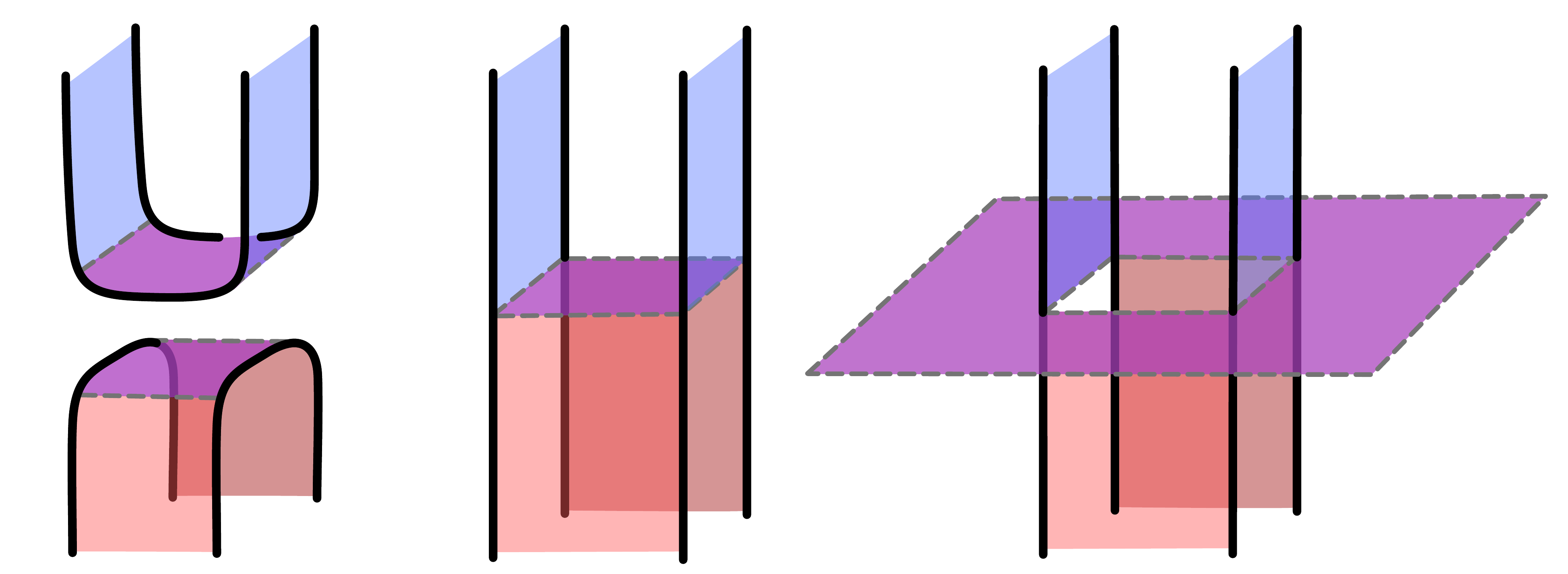}[font=\small]
    \node[] at (17.8,27.4) {$S_1$};
    \node[] at (8.25,4.6) {$S_2$};
    \node[left] at (4.4,23.9) {$L_1$};
    \node[left] at (4.6,7.8) {$L_2$};
    \node[left=2pt] at (31.5,29.2) {$L$};
    \node[] at (45.7,27.4) {$S$};
    \node[] at (39.1,18.5) {$P$};
    \node[left=2pt] at (66.5,29.2) {$L$};
    \node[] at (80.8,27.4) {$S'$};
    \node[] at (91.2,22.2) {$S^2 \setminus P$};
    \node[below=5pt] at (8.8,0.4) {(a) The initial surfaces};
    \node[below=5 pt] at (37.7,0.4) {(b) The surface $S$};
    \node[below=5 pt] at (73.1,0.4) {(c) The surface $S'$};
\end{tikzoverlay}
  \caption{The Murasugi sum with $k = 2$.}
  \label{fig:murasugi}
\end{figure}
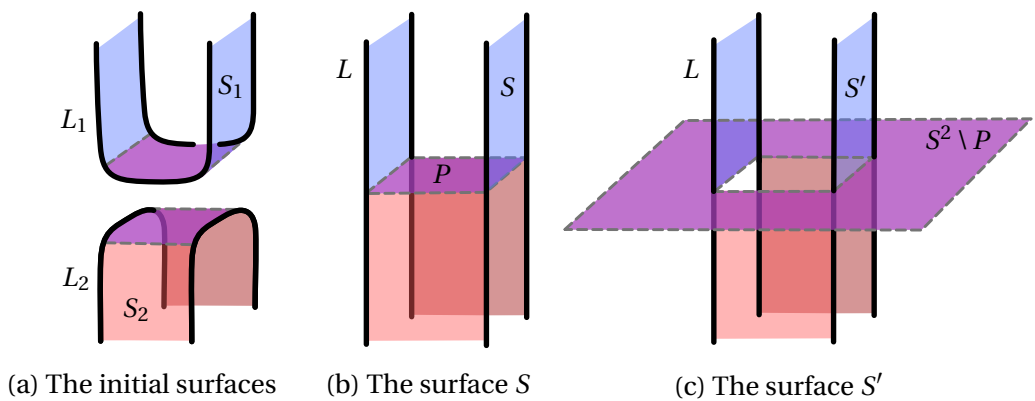

Gabai showed that if each $S_i$ is minimal genus, then so is
$S$; similarly if each $S_i$ is a fiber, then so is $S$ \cite{Gabai1983}.
We generalize these results to:

\begin{lemma} \label{lem:murasugi} If $S_1$ and $S_2$ are libroid surfaces, and
  $S$ is obtained from $S_1$ and $S_2$ by Murasugi sum, then $S$ is also a
  libroid surface. 
\end{lemma}

\begin{figure}[tb]
\begin{tikzoverlay}[width=9cm]{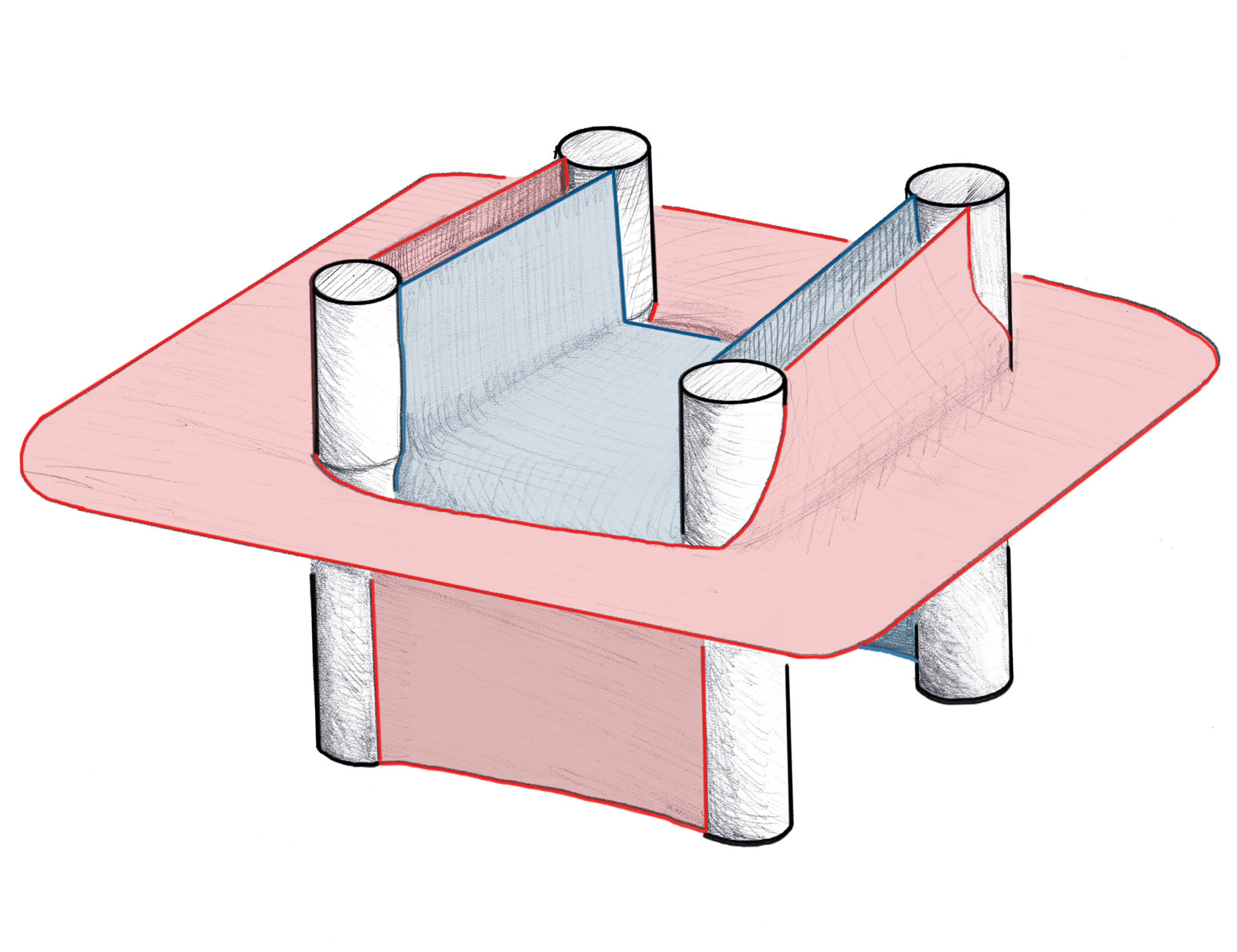}[font=\small]
    \node at (50.3,44.6) {$S$};
    \node at (15.0,42.0) {$S'$};
    \begin{scope}[line width=1pt, line cap=round, ->]
    \draw (60.4,72.6) .. controls  (50.8,74.5) and  (42.3,69.5) .. (43.9,61.0);
    \node[right] at (59.7,72.0) {$R_2$};
    \draw (66.3,71.2) .. controls (69.2,70.0) and (72.0,63.8) .. (72.0,56.5);

    \draw (74.0,8.8) .. controls (65.2,3.2) and (49.5,4.0) .. (49.4,10.5);
    \draw (74.0,13.1) .. controls (70.7,16.9) and (70.4,20.8) .. (70.5,23.6);

     \node at (77.0,10.8) {$R_1$};
    \end{scope}
\end{tikzoverlay}
  \caption{The surfaces $S$ and $S'$ made disjoint.}
  \label{fig:libroid}
\end{figure}
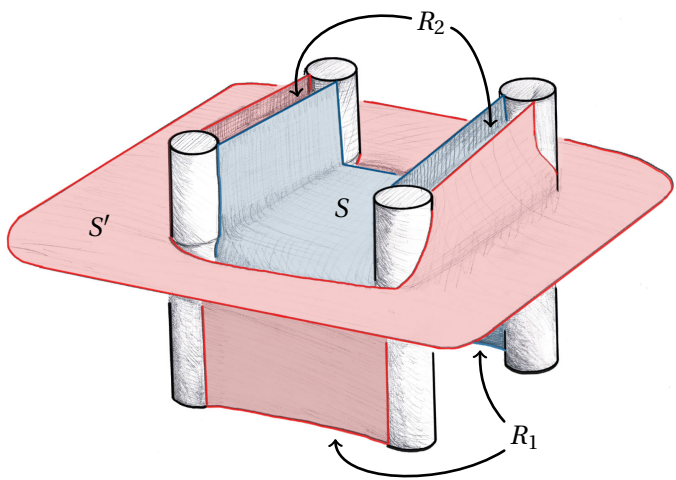

\begin{proof}
  The Seifert surfaces $S$ and $S'$ for $L$ can be disjointly embedded
  as sketched in Figure~\ref{fig:libroid}.  In detail, take a regular
  neighborhood $N(L)$, and form the exterior
  $E(L)=\overline{S^3-N(L)}$. We'll use the notation above for
  Murasugi sum. Then $S^2\cap E(L)$ is a $2k$-punctured sphere,
  dividing $E(L)$ into tangle complements $T_i= E(L)\cap B_i$.  Take a
  regular neighborhood $R_{3-i}$ of $\overline{S_i-P}\cap T_i$ inside
  $T_i$; then the relative boundary of $R_{3-i}$ in $T_i$ is two parallel copies of
  $\overline{S_i-P}$. The union with $S^2-(R_1\cup R_2)$ gives our two
  disjointly embedded Seifert surfaces $S\cup S'$.
  
  The complements $S^3-S_i$, $S^3-S$, and $S^3-S'$ naturally admit
  sutured manifold structures as described in Section 4 of
  \cite{Sakuma1994}.  Moreover, the two complementary regions
  $S^3-(S\cup S')$ may be identified with $(S^3-S_i) \cup R_i$, where
  $R_i$ is the product sutured manifold described above, and $R_i$ is
  attached to $S^3-S_i$ along $k$ product disks in the sutures
  corresponding to $R_i\cap S^2$ (recall $k$ is defined by
  $S_1\cap S_2=P$ is a $2k$-gon).  But $S^3-S_i$ is a library sutured
  manifold, which may be extended as products into $R_i$ to obtain a
  library decomposition of $S^3-(S\cup S')$.  Thus $S$ and $S'$ are
  libroid Seifert surfaces for $L$.
\end{proof}

\begin{remark}
  The sutured manifold decomposition in the above proof is the same as
  that in \cite[Condition 4.2]{Sakuma1994}; while we first decompose
  along $S\cup S'$ and then along the $2k$ product disks and remove
  the product sutured manifolds $R_i$, Sakuma first decomposes along
  $S$ and then along the disk $S^2-P$, resulting in the union of the
  sutured manifolds $S^3-S_i$.
\end{remark}

\begin{figure}[tbh]

\begin{tikzpicture}[font=\small]
  \begin{scope}[anchor=north west, inner sep=0]
    \node at (0, 9.5)  {\includegraphics[width=3cm]{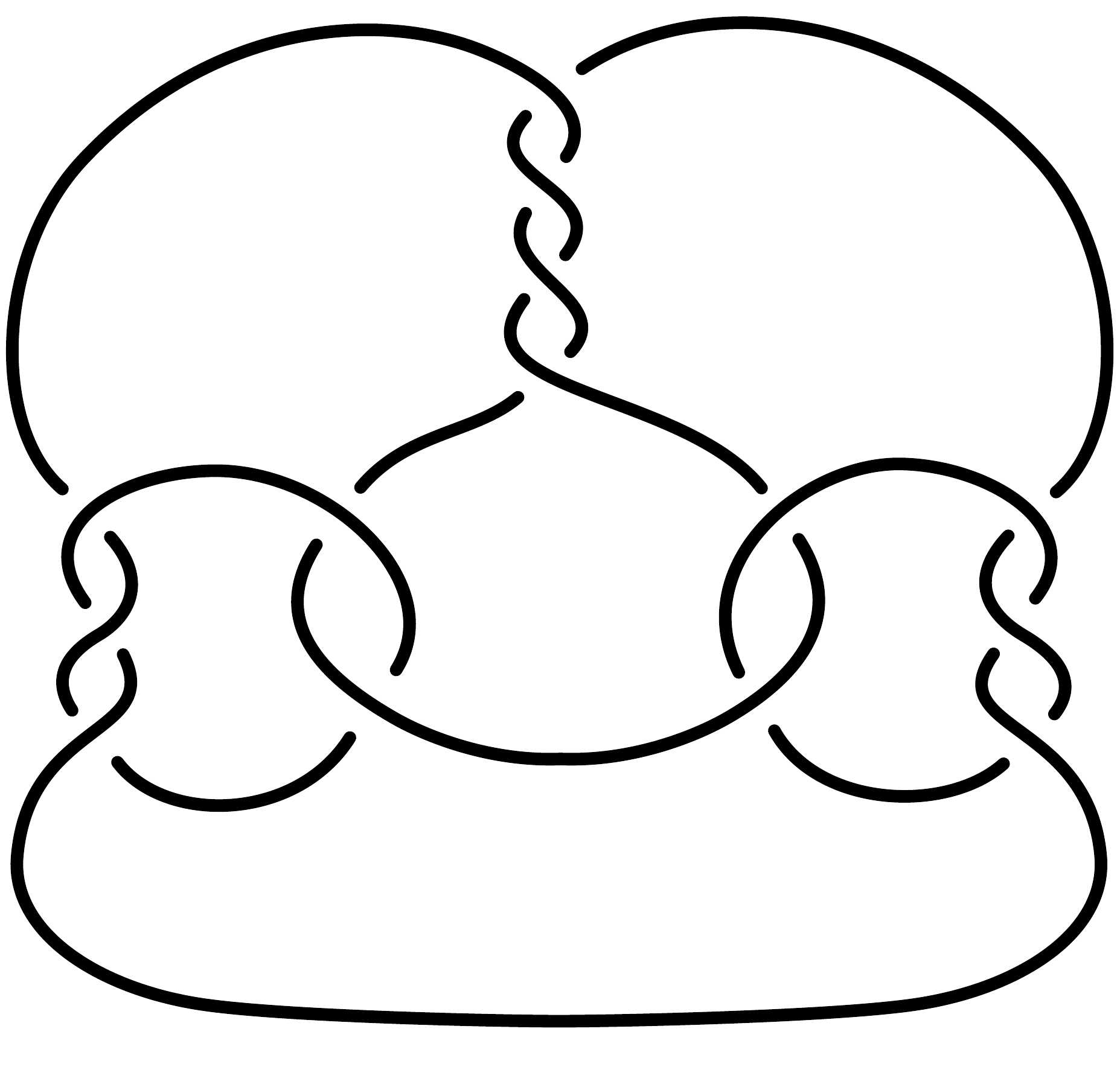}};
    \node at (4.5, 9)  {\includegraphics[width=3cm]{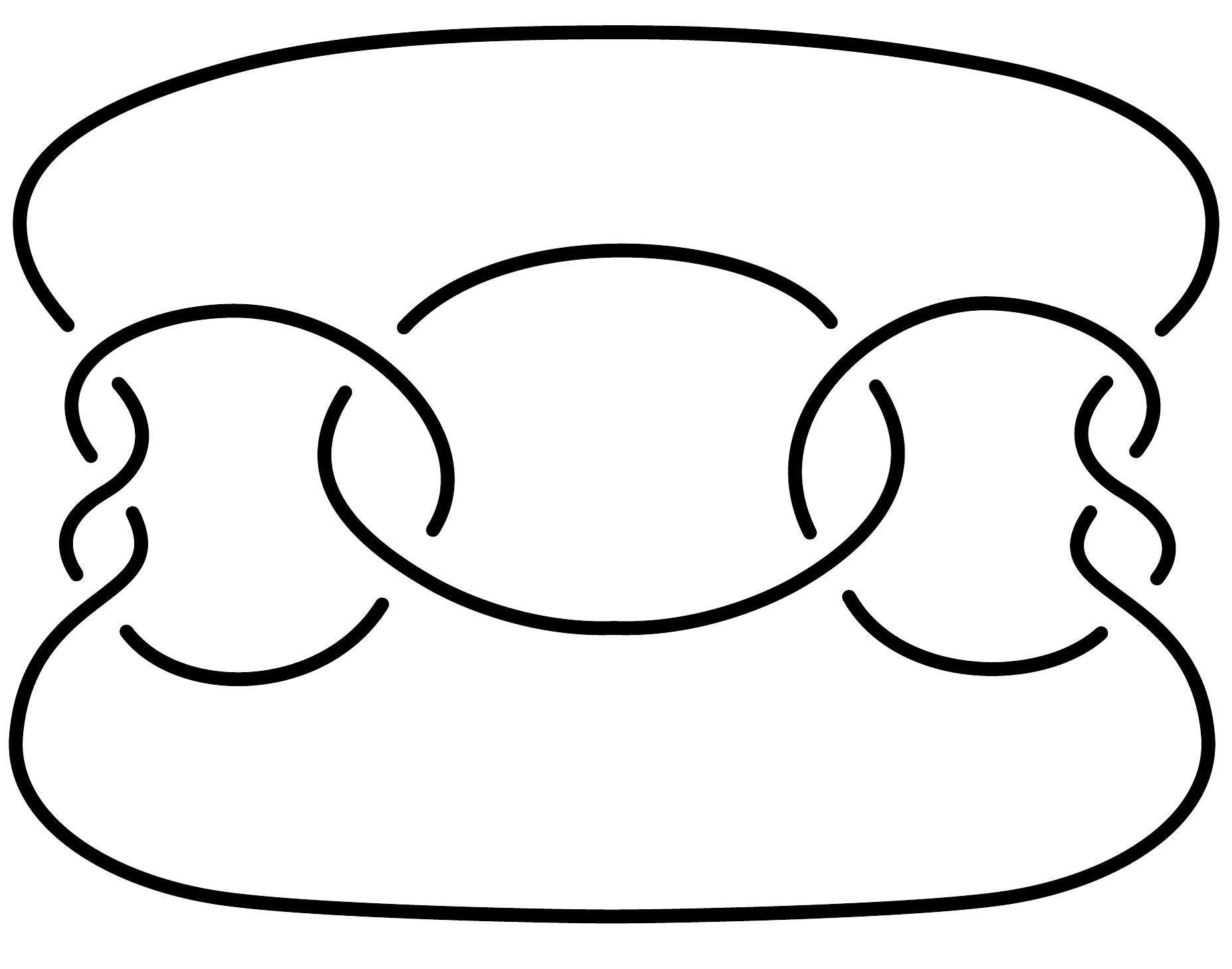}};
    \node at (9, 10) {\includegraphics[width=5cm]{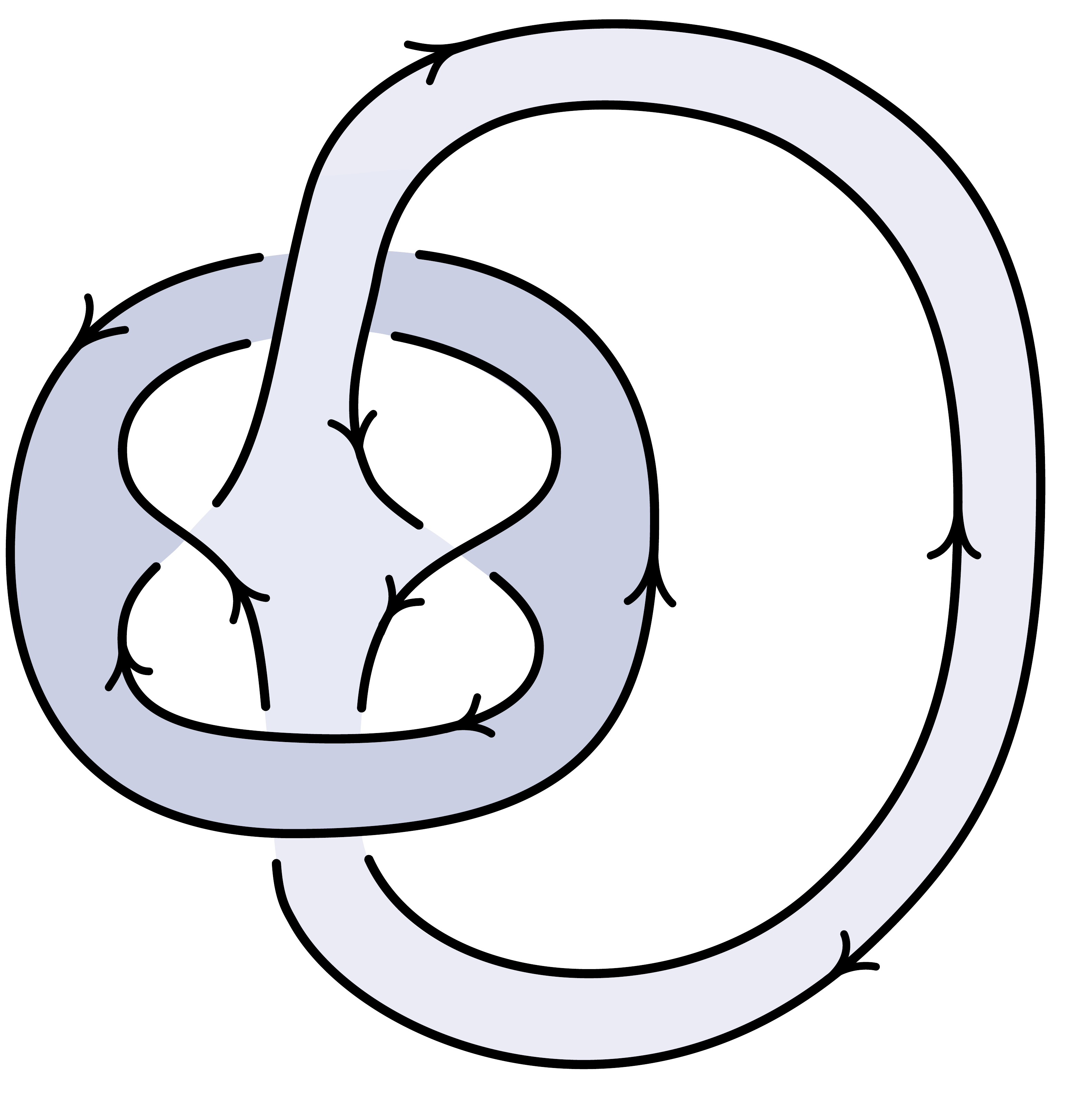}};
    \node at (0.3, 5.25) {\includegraphics[width=7cm]{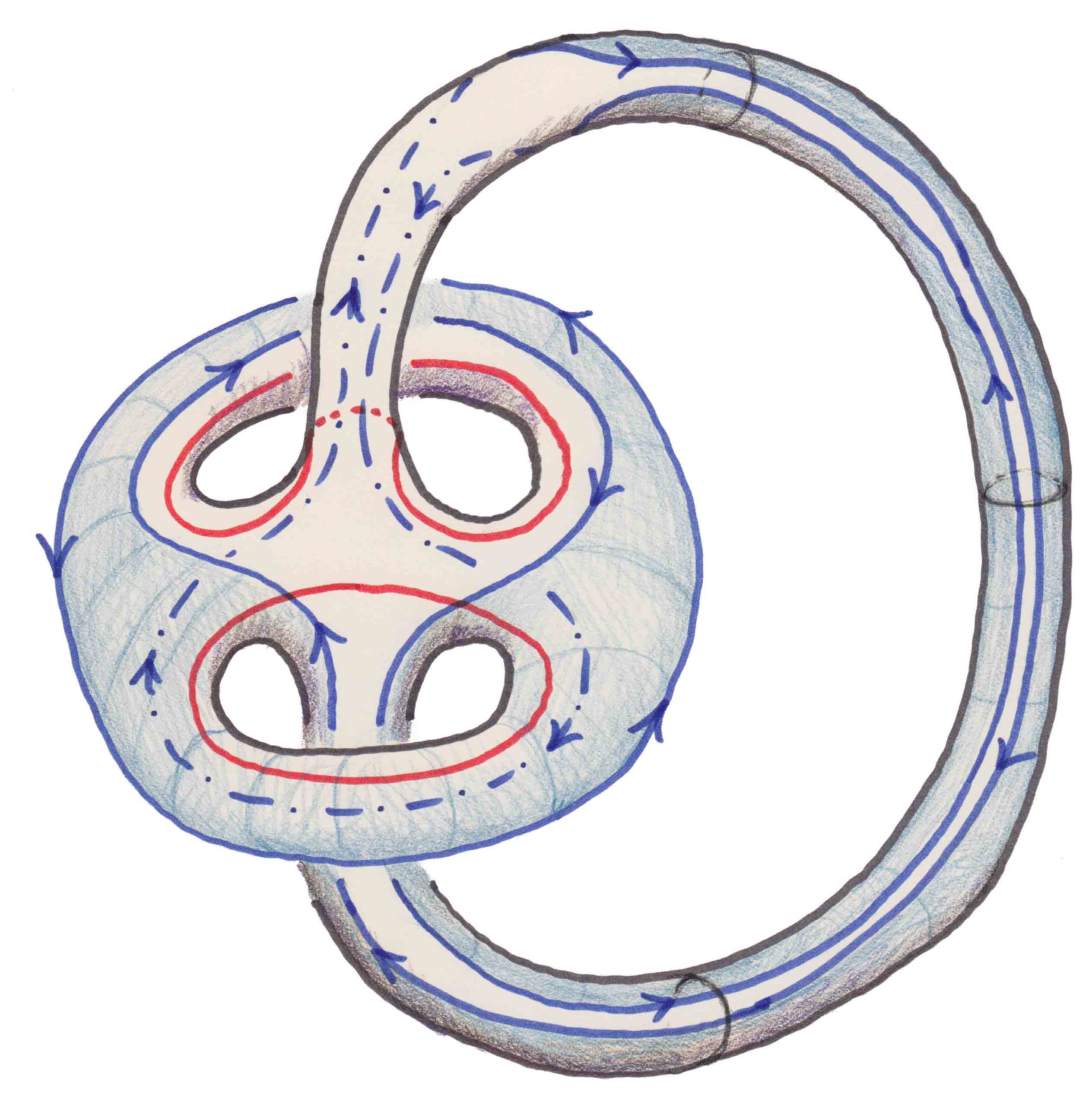}};
    \node at (8.6, 3.75) {\includegraphics[width=4.75cm]{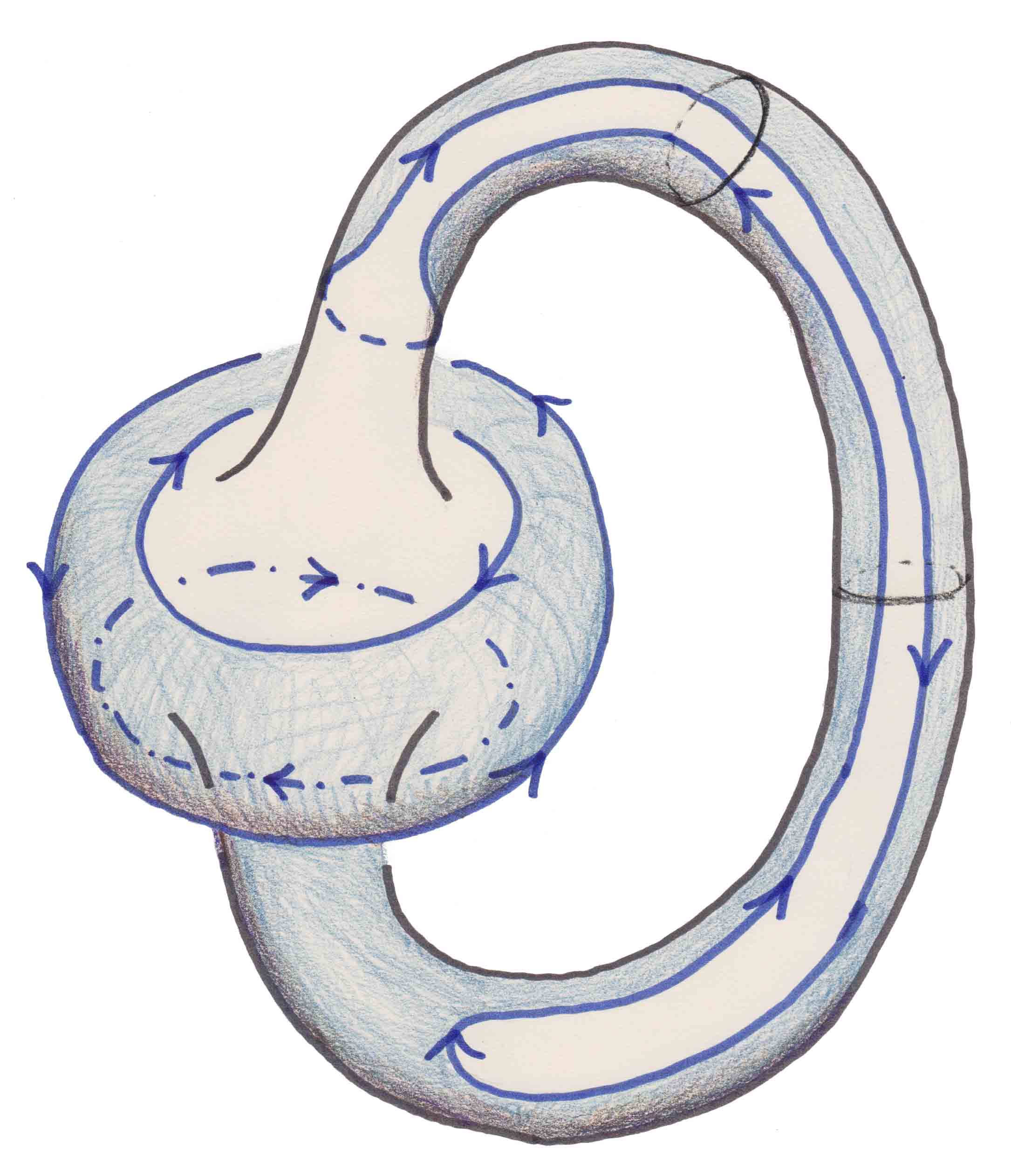}};
  \end{scope}
  \begin{scope}[line width=1pt, line cap=round, ->]
    \draw (9.25, 9.65) .. controls (9.55, 9.60) and (9.75, 9.1) .. (9.8, 8.85);
    \draw (7.1, 9.65)  .. controls (6.7, 9.60) and (6.3, 9.4) ..  (6.1, 9.1);
  \end{scope}
  
  \node at (1.5, 6.1) {(a)  The knot $\KT{2}{n}$.};
  \node at (6, 6.1) {(b)  The pretzel link $L$.};
  \node[right=3pt] at (1.5, 8.9) {$2n$};
  \node at (8.2, 9.7) {plumb here};
  \node at (11.5, 4.7) {(c) The surface $S$.};
  \node at (4, -2.2) {(d) $M$ viewed from inside.};
  \node at (11.5, -2.2) {(e) $T$ viewed from inside.};

\end{tikzpicture}

\caption{Kinoshita-Terasaka knots and the proof of
  Theorem~\ref{thm:KT}.}
\label{fig:KT}

\end{figure}
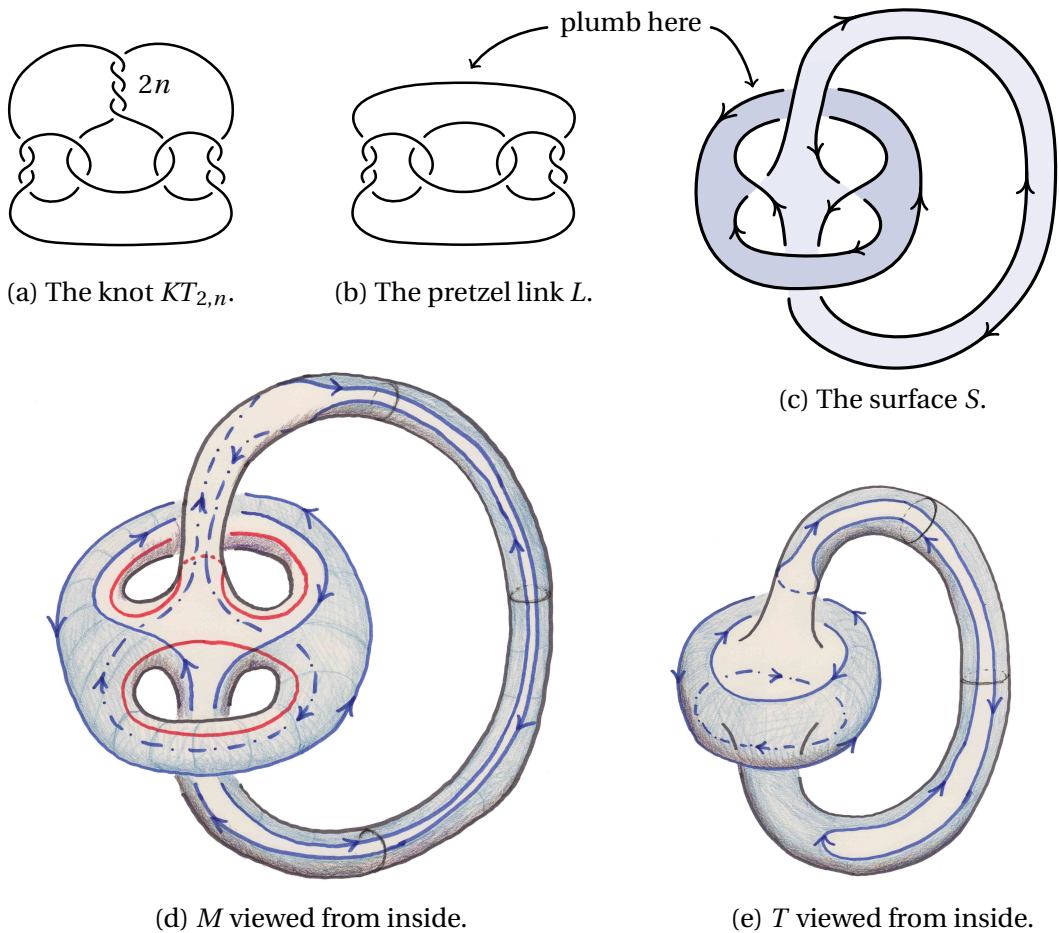

The class of \emph{arborescent links} are those obtained by
plumbing together twisted bands in a tree-like pattern (see
e.g.~\cite{Gabai1986, BonahonSiebenmann} for a definition).  It is
important to note that the bands are allowed to have an odd number of
twists.  With a few known exceptions, these links are hyperbolic (see
\cite{BonahonSiebenmann} or \cite[Theorem 1.5]{FuterGueritaud2009}).
The subclass of \emph{special arborescent links} studied by Sakuma
\cite{Sakuma1994} are those obtained by plumbing bands with even
numbers of twists, and hence the plumbed surface is a Seifert surface
for the link.  Inductively applying Lemma~\ref{lem:murasugi} shows
that all \emph{special} arborescent knots are libroid.  A famous
family of non-special arborescent knots are the Kinoshita-Terasaka
knots; to complete the proof of Theorem~\ref{thm:hyplib}, it suffices
to show:

\begin{theorem}\label{thm:KT}
  The Kinoshita-Terasaka knots $\KT{2}{n}$ shown in
  Figure~\ref{fig:KT}(a) are libroid hyperbolic knots with trivial
  ordinary Alexander polynomial.
\end{theorem}

\begin{proof}
  These knots are hyperbolic since they are arborescent and not one of
  the exceptional cases, and their Alexander polynomials were
  calculated in \cite{KinoshitaTerasaka1957}.  Minimal genus Seifert
  surfaces were found by Gabai \cite[\S 5]{Gabai1986}; we review
  his construction to verify that these knots are libroid.

  Let $L$ be the $(3, -2, 2, -3)$--pretzel link shown in
  Figure~\ref{fig:KT}(b); a Seifert surface $S$ for one orientation
  of $L$ is shown in Figure~\ref{fig:KT}(c).  The surface $S$ is a
  twice-punctured torus, and hence taut since $S^3 \setminus L$ is
  hyperbolic.  The $\KT{2}{n}$ knot can be obtained by plumbing a band
  with $2n$-twists onto $S$ in the location shown, so by
  Lemma~\ref{lem:murasugi} it suffices to prove that the complement of
  $S$ is a book of $I$-bundles.

  Thickening $S$ to a handlebody, we get the picture in
  Figure~\ref{fig:KT}(d); the outside of this handlebody is the
  sutured manifold $M$ we seek to understand.  Each short red curve
  meets the long blue oriented sutures in two points and bounds an
  obvious disk in $M$.  These are product discs in the sense of
  \cite{Gabai1986}, so we decompose along them to get the sutured
  manifold $T$ which is the \emph{exterior} of the solid torus shown
  in Figure~\ref{fig:KT}(e).  Note that $T$ is a solid torus with four
  sutures that each wind once around in the core direction.  In
  particular $T$ is taut and hence so is $M$; moreover, thinking
  backwards to build $M$ from $T$ by reattaching the product discs
  shows that $M$ is a book of $I$-bundles with a single binding which
  is basically $T$.
\end{proof}

{\RaggedRight \small
\bibliographystyle{nmd/math} 
\bibliography{\jobname}
}
\end{document}